\documentclass[12pt]{amsart}
\usepackage{amsfonts}
\usepackage{amsmath, amsthm, amssymb,color}
\usepackage{latexsym}
\usepackage{amsaddr}
\usepackage{graphicx}
\usepackage{graphics}
\usepackage{epsfig}
\usepackage[english]{babel}
\usepackage{paralist}
\usepackage{multirow}
\usepackage[T1]{fontenc}
\numberwithin{equation}{section}
\allowdisplaybreaks

\def\bfA{{\mathbf{A}}}
\def\bfB{{\mathbf{B}}}
\def\bfC{{\mathbf{C}}}
\def\bfD{{\mathbf{D}}}
\def\bfX{{\mathbf{X}}}

\def\bfm{{\mathbf{m}}}
\def\bfa{{\mathbf{\alpha}}}
\def\C{{\mathbb{C}}}
\def\E{{\mathbb{E}}}

\def\N{{\mathbb{N}}}
\def\P{{\mathbb{P}}}

\def\R{{\mathbb{R}}}

\def\Z{{\mathbb{Z}}}

\newcommand{\wt}{\widetilde}
\newcommand{\s}{\sigma}

\def\8{\infty}
\def\N{\mathbb{N}}
\def\E{\mathbb{E}}
\def\P{\mathbb{P}}

\def\<{\langle}
\def\>{\rangle}
\def\C{\mathbb{C}}

\renewcommand{\phi}{\varphi}
\renewcommand{\d}{\delta}
\renewcommand{\a}{\alpha}

\newcommand{\D}{\Delta}

\renewcommand{\L}{\Lambda}

\newcommand{\eps}{\varepsilon}

\newcommand{\g}{\gamma}

\renewcommand{\o}{\omega}

\newcommand{\supp}{\mathrm{supp}}
\newcommand{\Ind}{\mathbf{1}_}

\newtheorem{thm}[equation]{Theorem}

\newtheorem{lem}[equation]{Lemma}

\theoremstyle{definition}

\numberwithin{equation}{section}

\usepackage{ascmac}
\begin{document}
	\bibliographystyle{alpha}
	
	\title[Diagonal stochastic recurrence equation]{Stochastic recurrence equation with diagonal matrices}
	\today
	\author[E. Damek]{Ewa Damek}
	\address{Institute of Mathematics
				University of Wroclaw\\ Pl.  Grunwaldzki 2, 50-384, Wroclaw, Poland}  
	\email{ewa.damek@math.uni.wroc.pl}

		\begin{abstract}
						We consider the stochastic equation  $\bfX \stackrel{d}{=} \bfA  \bfX+\bfB $ where $\bfA $ is a random diagonal matrix and $\bfX,\bfB $ are random vectors, $\bfX, \bfA $ are independent and the equation is meant in law.  
											We prove that  $\bfX $ is regularly varying in a multivariate nonstandard sense. The results are applicable to stochastic recursions  
			with diagonal matrices, in particular, to multivariate autoregressive models like CCC-GARCH or BEKK-ARCH.\\
						\vspace{2mm} 
		\end{abstract}
	\maketitle

	\section{Introduction}
	We start with the stochastic recurrence equation (SRE)
	\begin{equation}\label{iterations}
	\bfX_n=\bfA_n\bfX_{n-1}+\bfB_n, \quad n\in \N,
	\end{equation}
	where $(\bfA_n,\bfB_n)$ is an i.i.d.\ sequence, $\bfA_n$ are $d\times d$
	matrices, $\bfB_n$ are vectors and $\bfX_0$ is an initial vector
	independent of the sequence $(\bfA_n, \bfB_n)$. 
	Under mild contractivity hypotheses \cite{bougerol:picard:1992a}, 
	the sequence $\bfX_n$ converges in law to a random
	vector $\bfX$. $\bfX$ is the unique solution to the equation
	\begin{equation}
	\label{affine} 
	\bfX \stackrel{d}{=} \bfA \bfX+\bfB, 
	\end{equation} where $(\bfA,\bfB)$ is a generic copy of $(\bfA_n,\bfB_n)$,  $\bfX$ is independent
	of $(\bfA,\bfB)$ and the equation is meant in law. If the process \eqref{iterations} is started from $\bfX _0=\bfX$ then it becomes stationary.
	
	There has considerable interest in various aspects of the
	iteration \eqref{iterations} and, in particular, the tail behavior of
	$\bfX$, \cite{alsmeyer:mentmeier2012}, \cite{buraczewski:damek:guivarch2009}, \cite{goldie:1991}, \cite{kesten:1973}.

	In this paper we focus on the case when the matrix $\bfA = diag (A_1,...,A_d)$ is diagonal and we make the story more general. Suppose that \eqref{affine} is satisfied but $\bfX$ and $\bfB$ are not necessarily independent and possibly there is no iteration behind. Assume however that $\bfX$ and $\bfA$ are independent. Are we still able to describe the tail of $\bfX$ and is it of  
	any interest anyway? It turns out that such a situation appears naturally when Gaussian multiplicative chaos is considered \cite{wong:2020}. 
	
	Suppose that \eqref{affine} is satisfied, $\bfA = diag (A_1,...,A_d)$ and $\bfX=(X_1,...,X_d)$ are independent. Under so called Kesten-Goldie conditions (\eqref{A0}-\eqref{A2}), we prove that $X$ is regularly varying in the sense of a nonstandard regular variation. More precisely, let $\bfB=(B_1,...,B_d)$, and, among other things, we assume that 
	\begin{equation}\label{kes-gol}
	\mbox{for every}\ j \ \mbox{there is}\  \a _j>0 \ \mbox{such that}\ \E |A_j|^{\a _j}=1. 
	\end{equation} 
	Then the behavior of marginals $X_j$ follows from \cite{goldie:1991}. Namely,
	for every $j$
	\begin{equation*}
	X_j\stackrel{d}{=}A_jX_j+B_j
	\end{equation*}
	and, under assumptions {\eqref{A0}-\eqref{A2}} below,
	\begin{equation}\label{golintro}
	\lim _{t\to \8}\P (\pm X_j>t)t^{\a _j}=c_{j,\pm},
	\end{equation}
	see Theorem \ref{Goldie1}. In fact \eqref{golintro} was proved in \cite{goldie:1991} only when $(A_j,B_j)$ and $X_j$ are independent but all what is needed is the Goldie Implicit Renewal Theorem \ref{Goldie} 
	and only o slight modification is needed to get it in our framework, see Theorem \ref{Goldie1}. 
	
	Since $\a _1,...,\a _d$ may be different, we cannot expect the standard multivariate regular variation i.e. existence of the weak limit of measures
	\begin{equation}\label{standardreg}
	\P (\| \bfX\| ^{-1}\bfX\in U \ \ |\ \ \| \bfX\| >t)
	\quad \mbox{as}\  t\to \8,
	\end{equation}
	Here $\| \ \| $ is the Euclidean norm and $U $ denotes a Borel subset of the sphere. Instead, we need to introduce an appropriate norm $|\cdot |_{\a}$ and an appropriate scaling $\d _t$ as it is suggested in \cite{mentemeier:wintenberger:2021}: 
	\begin{equation}\label{dilat}
	\d _t(x)= (t^{1\slash \a _1}x_1,...,t^{1\slash \a _d}x_d ).
	\end{equation}
	Notice that, if $c_{j,\pm}>0$ then by \eqref{golintro}
	\begin{equation*}
	\P (\pm X_j> t^{1\slash \a _j})\sim c_{j,\pm}t^{-1}
	\end{equation*}
	i.e. the tails of the marginals $X_j$ normalized this way behave as $t^{-1}$ and so it makes sense to ask for 
	\begin{equation}\label{regular}
	\lim _{t\to \8}t\P (\d _{t^{-1}}\bfX\in \cdot \ ),
	\end{equation}
	where $\cdot $ stands for a Borel subset of $\R ^d\setminus \{ 0\}$. 
	We prove that the above limit exists, Theorem \ref{main}, defines a Radon measure $\L $ on $\R ^d\setminus \{ 0\}$ and implies regular variation in the nonstandard sense. 
	For the origin of non standard regular variation \eqref{regular} we refer the reader to \cite{resnick:2007}, section 6.5.6 and further development in \cite{pedersen:wintenberger}. See also section 4.4.10 in \cite{buraczewski:damek:mikosch:2016}. 
	
	
	\subsection{Matrix iterations}
	Stochastic iterations \eqref{iterations}   
	have been studied since the seventies and they 
	found numerous applications to financial time series
	models 
	\cite[Section 4]{buraczewski:damek:mikosch:2016}, \cite{horvath:boril:2016}, \cite{starica:1999}, \cite{sun:chen:2014} as well as to risk management \cite[Sec. 7.3]{mcneil:fre:embrechts:2015}), \cite{das:hartman:kluppelberg:2022}. Recently asymptotic independence of regularly varying marginals have been considered in the case of neural networks \cite{luo:wang:fung:2020} and methods coming from SREs (\cite{buraczewski:collamore: damek:zienkiewicz:2016, kevei:2016, mentemeier:wintenberger:2021}) have been adopted. 
	
	The first set of conditions implying regular behavior of $\bfX$ in the sense of \eqref{directions} below 
	was formulated by Kesten \cite{kesten:1973} for matrices with positive entries. 
	Since then, Kesten conditions and their extensions have been used 
	to characterize tails in various situations, 
	an essential feature being 
	the same tail behavior in all directions \cite{alsmeyer:mentmeier2012, buraczewski:damek:guivarch2009, guivarch:lepage:2016}. To put it simply, it means that there is a Radon measure on $\R ^d\setminus \{ 0\}$ being the weak limit of
	\begin{equation}\label{directions}
	t^{\a}\P (t^{-1}\bfX \in \cdot ), \quad \mbox{when} \quad t\to \8,
	\end{equation}
	for some $\a >0$. \eqref{directions} follows from certain
	irreducibility or homogeneity of the action of the group generated by
	the support of the law of $\bfA$. However, this property is not necessarily shared by all models interesting both from theoretical and applied perspective \cite{matsui:mikosch:2016}, \cite{matsui:pedersen:2022}, \cite{mentemeier:wintenberger:2021}, \cite{pedersen:wintenberger}. In particular, it does not hold for diagonal matrices which constitute usually the simplest model for applications. Therefore, SREs with such $\bfA$'s are both challenging and desirable.  
	
	Recently  regular variation of the stationary solution $\bfX$ to \eqref{iterations} has been proved in \cite{mentemeier:wintenberger:2021} for the case when 
	\begin{equation}\label{restrictive}
	\bfA=(b_1+c_1M,...,b_d+c_dM) 
	\end{equation}
	where $b_j,c_j$, $j=1,...,d$ are constants and the random variable $M$ is independent of $\bfB$.
	
	The restrictive nature \eqref{restrictive} and the absence of any other results on diagonal matrices shows that new methods are required and they are provided in this paper. The key observation in \cite{mentemeier:wintenberger:2021} is that, under \eqref{restrictive}, if
	$\a _i\neq \a _j$ then $X_i, X_j$ are asymptotically independent i.e.
	\begin{equation}\label{indenpendence}
	\P (|X_i|>t^{1\slash \a _i}, |X_j|>t^{1\slash \a _j}) =o\left ( \P (|X_i|>t^{1\slash \a _i})\right )=o(t^{-1})\quad \mbox{as}\ t\to \8.
	\end{equation}
	We are able to prove \eqref{indenpendence} in full generality i.e. under assumption 
	\begin{equation*}
	\P \left (|A_i|^{\a _i}=|A_j|^{\a _j}\right )<1,
	\end{equation*}
	and when only $\bfA$ and $\bfX$ are independent in \eqref{affine} i.e. we do not require iterations, see Section \ref{indep}. 
	
	\subsection{Gaussian multiplicative chaos}
	In the theory of Gaussian multiplicative chaos a random measure $M_{\gamma} $ on $\R ^d$ is studied \cite{wong:2020}. Given an open set $W\subset \R ^d$, asymptotics of $\P (M_{\g }(W)>t)$ as $t\to \8 $ is of interest. To prove that 
	$$\P (M_{\g }(W)>t)\sim t^{-\a},\ \mbox{as}\ t\to \8 $$
	one dimensional equation $X=AX+B$ a.s. with $X,A$ independent but $X,B$ being dependent is used in \cite{wong:2020} and Goldie Implicit Renewal Theorem \ref{Goldie} is applied.  
	
	It seems that also multivariate \eqref{affine} with dependent $\bfX, \bfB$ may be of interest while Gaussian multiplicative chaos is considered. The question has been ask to the author by Tomas Kojar and resulted 
	in removing independence of $\bfX$ and $\bfB$ from the picture. 
	  
	\subsection{Applications to financial models} 
	There are various financial models that satisfy \eqref{iterations} and \eqref{affine} with various assumptions on $\bfA$ and $\bfB$. Then, to prove that the stationary solution $\bfX$ is regularly varying, the results of \cite{alsmeyer:mentmeier2012, buraczewski:damek:guivarch2009, guivarch:lepage:2016, kesten:1973} have been used in  \cite{basrak:davis:mikosch:2002, matsui:mikosch:2016, matsui:pedersen:2022,  pedersen:wintenberger}. However they do not cover the most desirable case from the point of view of applications: diagonal $\bfA$ with arbitrary covariance matrix $[cov (A_i,A_j)]$ and possibly different tail indices $\a _1,...,\a _d$. Our result fills the gap. When Theorem \ref{main} is proved, the analysis of so-called spectral process may be pursued as in \cite{mentemeier:wintenberger:2021}.   
	
	Suppose, we consider a number of financial assets - a vector  such that the component-wise returns are described by univariate GARCH(1,1) processes. Then the squared volatilities of components satisfy one dimensional SRE \eqref{iterations} and, under suitable assumptions, they are regularly varying but with possibly different tail indices. The returns inherit regular variation. It is natural to ask for the joint regular variation under normalization  \eqref{regular}. 
	
	Regular variation of $\bfX$ is a convenient starting point to study extremal behavior of the process in terms of the maxima and extremal indices, see e.g \cite{pedersen:wintenberger}. This helps to understand the most and the least risky assets in the above example and, generally, to study conditional value at risk and risk measures \cite{das:hartman:kluppelberg:2022}.  It is possible to obtain the exremogram \cite{matsui:mikosch:2016}, point process convergence \cite{buraczewski:damek:mikosch:2016, pedersen:wintenberger} as well as stable limit theory, see Section 4.5 of \cite{buraczewski:damek:mikosch:2016} or \cite{buraczewski:damek:guivarch2009,pedersen:wintenberger}. Nothing like that has been studied so far for the process \eqref{iterations} with diagonal matrices $\bfA _n$ and now our results make it possible.  
	
	
	The BEKK-ARCH process, introduced by Engle and Kroner \cite{engle:kroner:1995} and originally defined by a non-affine recursion, has been written as \eqref{iterations} by Pedersen and Wintenberger \cite{pedersen:wintenberger}. They studied the regular behavior when assumptions of \cite{buraczewski:damek:guivarch2009} or \cite{alsmeyer:mentmeier2012} are applicable. \cite{pedersen:wintenberger}, however, does not cover diagonal BEKK models typically used in 
	finance due to their relatively simple parametrization, as discussed in Bauwens et al. \cite{bauwens:laurent:rombouts:2006}. This was probably  the main motivation of Mentemeier, Wintenberger to study diagonal SRE models and to apply their result to particular cases of BEKK-ARCH and CCC-GARCH \cite{mentemeier:wintenberger:2021}. Also BEKK-ARCH with triangular matrices has been of interest, \cite{matsui:pedersen:2022} but then only regular behavior of components $X_j$ follows from \cite{damek:matsui:2022}, \cite{damek:matsui:swiatkowski:2019}, \cite{damek:zienkiewicz:2017}, \cite{matsui:swiatkowski}.
	

	\section{Main results} 
	Recall that $\bfA =diag (A_1,...,A_d)  $, $\bfB =(B_1,...,B_d)\in \R ^d$,
	\begin{equation}
	\label{affine1} 
	\bfX \stackrel{d}{=} \bfA \bfX+\bfB
	\end{equation}
	and $\bfA ,\bfX$ are independent.
	We assume that for every $j$, 
	\begin{equation}\label{A0}
	\log |A_j| \quad \mbox{conditioned on}\quad A_j\neq 0 \quad\mbox{is non arithmetic}.
	\end{equation}

	\begin{equation}\label{A1}
	\mbox{there is}\ \a _j>0\ \mbox{such that}\ \E |A_j| ^{\a _j}=1
	\end{equation}
	
	\begin{equation}\label{A2}
	\E |A_j| ^{\a _j }\log ^+|A_j|<\8 \ \mbox{and there is}\ \s>0\ \mbox{such that },\  \E |B_j| ^{\a _j+\s }<\8.
	\end{equation}
	\eqref{A0}, \eqref{A1} imply that
	\begin{equation}\label{neglog}
	-\8 \leq \E \log |A_j|<0
	\end{equation}
	and in view of \eqref{A1}, \eqref{A2} 
	\begin{equation}\label{posder}
	0<  \E |A_j|^{\a _j}\log |A_j|<\8 ,
	\end{equation}
	see Theorem \ref{Goldie}.
	
	\medskip
	We are going to use the same norm as in \cite{mentemeier:wintenberger:2021} i.e.
	\begin{equation*}
	|x|_{\bfa }=\max _{1\leq j\leq d}|x_j|^{\a _j}\quad \mbox{for}\ x\in \R ^d\end{equation*}
	$|\cdot |_{\a }$ may not be subadditive but there is always $c_{\a}\geq 1$ such that
	\begin{equation}\label{subad}
	|x+y|_{\a }\leq c_{\a }\left (|x|_{\a}+|y|_{\a}\right ).
	\end{equation}
	For $\d _t$ defined in \eqref{dilat}, we have
	\begin{equation*}
	|\d _t(x)|_{\a} = t|x| _{\a}.
	\end{equation*}
	Finally, let
	\begin{align*}
	S_r^{d-1}=\{ x\in \R ^d: |x| _{\a }=r\}, &\quad S^{d-1}:=S_1^{d-1}\\
	B_r(0)=\{ x\in \R ^d: |x| _{\a }<r\}, &\quad B_r(0)^c=\{ x\in \R ^d: |x| _{\a }\geq 1\}.
	\end{align*}
	
	\begin{thm}\label{mainth}
		Suppose that \eqref{affine1}-\eqref{A2} are satisfied,  
		Then 
		\begin{equation*}
		\lim_{t\to \8}\P (|\bfX|_{\a }>t)=c_{\8 }
		\end{equation*}
		exists. If $c_{\8 }>0$ the sequence of measures 
		\begin{equation}\label{eqmain}
		m _t(U)=\P \left (\d _{|\bfX|_{\a }^{-1}}\bfX\in U \ \ \Big |\ \ |\bfX|_{\a }>t\right )
		\end{equation}
		defined on $S ^{d-1}$ converges weakly to a non zero probability measure $\nu $ when $t\to \8$.
	\end{thm}
	{\bf Remark.} Theorem \ref{mainth} is proved in Section \ref{conclusion}. 
	If $\bfX$ and $\bfB$ are independent then $c_{\8 }>0 $ if and only if $\P (\bfA x+\bfB =x)<1$
	for every $ x\in \R ^d$. Strict positivity of $c_{\8 }$ when $\bfX$ and $\bfB$ are possibly dependent is discussed below Theorem \ref{main}. 

\medskip
	The scheme of the proof is as follows. For $i,j\in \{ 1,...,d\}$ we define an equivalence relation by
	\begin{equation*}
	i\sim j\quad \mbox{if and only if }\ \ |A_j|^{\a _j}=|A_i| ^{\a _i } \quad \mbox{a.s}. 
	\end{equation*}
	with associated equivalence classes $I_l$, $l=1,...,p$ which without loss of generality may be chosen as 
	$I_1=\{1,...,d_1\}, I_2=\{d_1+1,...,d_1+d_2\}$, etc. 
	Then, 
	we may write
	\begin{equation*}
	\R ^d=\R ^{d_1}\times ...\times \R ^{d_p}, \quad \mbox{where} \ \R ^{d_l}\ \mbox{corresponds to}\ I_l,
	\end{equation*}
	and the stationary solution $\bfX= (\bfX^{(1)},..., \bfX^{(p)})$ to \eqref{affine} satisfies
	\begin{equation}\label{homcase}
	\bfX^{(l)}\stackrel{d}{=} \bfA ^{(l)}\bfX^{(l)}+\bfB ^{(l)}, \quad l=1,...,p,
	\end{equation} 
	with $\bfA ^{(l)}= diag (A_{d_1+...+d_{l-1}+1},..., A_{d_1+...+d_l})$, $\bfB ^{(l)} =(B_{d_1+...+d_{l-1}+1},..., B_{d_1+...+d_l})$. We are going to refer to \eqref{homcase} as {\it the homogeneous case}, because then $\bfA^{(l)}$ acts by homogeneous dilations, see \eqref{norm}. 
	
	In the next section we will prove (Theorem \ref{hom}) that
	\begin{equation*}
	\lim _{t\to \8}t\P (|\bfX^{(l)}|_{\a}>t)=c_l
	\end{equation*}
	and if $c_l>0$
	 then the sequence of measures
	\begin{equation}\label{mainl}
	\P \left ( \d _{|\bfX^{(l)}|_ {\a }^{-1}}\bfX^{(l)}\in U  \ \Big |\  |\bfX^{(l)}|_{\a }> t\right ),
	\end{equation}
	$U$ a Borel subset of $S^{d_l -1}$, tends weakly, as $t\to \8$, to a probability measure $\nu _{l}$. 
	
	
	Secondly, in Section \ref{indep}, we consider the case $i\nsim j$ and we obtain, for $r_1, r_2 >0$, 
	\begin{equation*}
	\P \left ( |X_j|>r_1 t^{1\slash \a _j}, |X_i|> r_2 t^{1\slash \a _i}\right )=o(t^{-1}), \quad \mbox{as}\ t\to \8,
	\end{equation*}
	which implies that for $l\neq q$
	\begin{equation}\label{difblocks}
	\P \left ( |\bfX^{(l)}|_{\a }> r_1t, |\bfX^{(q)}|_{\a }> r_2t\right )=o(t^{-1})\ \mbox{as}\ t\to \8.
	\end{equation}
	We are also able to say something more about the speed in \eqref{difblocks}, see section \ref{indep}.
	
	Then identifying $S^{d_l -1}$ with 
	\begin{equation*}
	\wt S^{d_l -1}=\{ x\in \R ^{d}: |x|_{\a }=1, x_i=0\ \mbox{for}\  i \notin I_{l}\}
		\end{equation*} and proceeding as in the proof of Theorem 6.1 in \cite{mentemeier:wintenberger:2021}, we conclude Theorem \ref{mainth} with the limit measure $\nu $ 
	having support in $\bigcup _{1\leq l\leq p}\wt S^{d_l-1}$, Section \ref{conclusion}.
	
	\eqref{eqmain} may be formulated also in an alternative way that gives some more insight into
	the tail behavior of $\bfX$. For that we
	introduce polar coordinates related to the norm $|\cdot | _{\a }$ using the map $\Phi : \R ^+\times S^{d-1}\to \R ^d\setminus \{ 0\}$,
	\begin{equation}\label{polar} 
	\Phi (s,\o )=\d _s\o \quad s>0,\ \o \in S^{d-1} 
	\end{equation} 
	which is a homeomorphism. A straightforward proof of that is contained in the Appendix.
	
	Let $\bfC (\R ^d)$ be the space of {\it bounded continuous functions} supported away from zero. 
	More precisely, 
	\begin{equation*}
	f\in \bfC (\R ^d)\ \mbox{then there is }\ r >0\ \mbox{such that}\ \supp f\subset \R ^d\setminus B_r(0).
	\end{equation*}
	We shall write
	\begin{equation*}
	\| f\| _{\8} = \sup _{ x\in \R ^d }|f(x)|.
	\end{equation*}
	Let $\lambda _t$ be a measure on $\R ^d \setminus \{ 0\}$ defined by 
	\begin{equation}\label{lambda}
	\lambda_t (W)=t\P \left (\bfX\in \d _{t^{-1}}W\right ), \quad \mbox{for a Borel set}\ W\subset \R ^d \setminus \{ 0\}.
	\end{equation} 
	We are going to prove that the sequence of measures $\lambda_t$ tends in a weak sense to a Radon measure $\Lambda $ on $\R ^d\setminus \{ 0\}$. More precisely,
	for $f\in \bfC (\R ^d)$ we have
	\begin{equation*}
	\langle f, \lambda _t\rangle \to \langle f, \Lambda \rangle \ \mbox{as}\ t\to \8.
	\end{equation*}

	Now we may summarize what was said above and formulate our main theorem. 
	\begin{thm}\label{main}
		Suppose that \eqref{affine1}-\eqref{A2} are satisfied.
		Let $f\in \bfC (\R ^d)$, $\R ^d = \R ^{d_1}\times ...\times \R ^{d_p}$. Then there are Radon measures $\L _l$ on $\R ^{d_l}\setminus \{ 0\}$ such that
		\begin{equation}\label{main1}
		\lim _{t\to \8} t\E f(\d _{t^{-1}}\bfX)=\sum _{l=1}^p\langle f_l, \L _l\rangle =\langle f, \L \rangle,
		\end{equation}
		where $f_l(x_l) =f (0,...,0,x_l,0,...,0)$. In particular, for every $l$,
		\begin{equation}\label{main3}
		c_l=\lim _{t\to \8}\P (|\bfX ^{(l)}| _{\a }>t)t
		\end{equation}
		exists. Moreover, there is a finite measure $\wt \nu _l$ on $S^{d_l -1}$ such that
		\begin{equation}\label{main2}
		\langle f_l, \L _l\rangle =  \int _0^{\8}\int _{S^{d_l-1}}f_l(\d _s\omega )\frac{ds}{s^2}d\wt \nu _l(\omega),\end{equation}
		$\wt \nu _l(\R ^{d_l})=c_l$, 
		  and
						\begin{equation}\label{main5}
		\lim _{t\to \8}\P \left ( |\bfX |_{\a }>t\right )t=\sum _{l=1}^pc_l=:c_{\8}.
		\end{equation}
		Finally,  if $c_{\8}>0$ then for the measure $\nu $ defined in Theorem \ref{mainth} we have
		\begin{equation}\label{decomposition}
						\nu =c^{-1}_{\8}\sum _{l=1}^p\d ^{(1)}_0\times ...\times \d ^{(l-1)}_0\times \wt  \nu _l\times \d ^{(l+1)}_0\times ...\times \d ^{(p)}_0,\end{equation} 
		where $\d ^{(l)}_0$ is the delta measure concentrated at zero in $\R ^{(l)}$.
	\end{thm}
		{\bf Remark.} The proof of Theorem \ref{main} is contained in section \ref{conclusion}.
		Observe that $\nu $ is supported by $\bigcup _{1\leq l\leq p}\wt S^{d_l-1}$.
		$\wt \nu _l$ may be zero (if so is $\L _l$) but, if at least for one $l$, $\wt \nu $ is not zero then $\nu $ is well defined.
		
		$c_{\8}>0$ if and only if $\lim _{t\to \8}t\P \left ( |X_j|>t^ {1\slash \a _j}\right )>0$ for at least one $j$. In the case when $\bfX$ and $\bfB$ are independent the latter is equivalent to $\P (A_jx_j+B_j=x_j)<1$ for every $x_j\in \R $, \cite{goldie:1991}, and so, $\L \neq 0$ if and only if $\bfX$ is not concentrated at a point.
		
		In the general case, in Theorem \ref{Rcontra} we give an equivalent condition that may be checked for individual models by using their specific characteristics.

		
	\medskip
		{\bf Remark.} For matrices $\bfA$ with strictly positive entries asymptotics of $\P (|\bfX|>t)$  was studied in \cite{buraczewski:damek:2010} but for a norm that is not equivalent to the one considered here. 
		Although the results of \cite{buraczewski:damek:2010} concern a more general case of diagonal action on nilpotent Lie groups, some of the observations made there have been used in the present paper.  
		 

\subsection{Examples}

\medskip
	{\bf Constant Conditional Correlation GARCH(1,1).} 
	
	\noindent For CCC-GARCH(1,1), the volatility vector satisfies SRE \eqref{iterations} with generic
$$
\bfA =(a_1 Z_1^2+b_1,...,a_d Z_d^2+b_d ),$$
where $a_i>0, b_i\geq 0$ and $Z_1,...,Z_d$ are $N(0,1)$ variables with arbitrary correlations. It is not difficult to see that then $|a_i Z_i^2+b_i|^{\a _i}=|a_j Z_j^2+b_j|^{\a _j}$ a.s. if and only if $\a _i=\a _j$, $a_i=a_j$, $b_i=b_j$ and $Z_i^2=Z_j^2$ a.s. Therefore, for the blocks $\bfX^{(l)}$ in this case, we have
\begin{equation*}
\bfX^{(l)}\stackrel{d}{=}A\bfX^{(l)}+\bfB,
\end{equation*}	
where $A\bfX^{(l)}$ means multiplication of $\bfX ^{(l)}$ by a positive scalar random variable $A$. Components $X_i$, $X_j$ driven by $A_i\neq A_j$ a.s. are asymptotically independent. 

\medskip


\medskip	
	\noindent {\bf BEKK-ARCH model. }
	
	\noindent Due to representation of BEKK-ARCH model described in \cite{pedersen:wintenberger}, the process satisfies SRE \eqref{iterations} with
	\begin{equation*}
	\bfA _n=\sum _{i=1}^lm_{i,n}\bfD _i,
	\end{equation*}
	where $\bfD _1,..., \bfD_l$ are deterministic matrices, the i.i.d process $\{ m_{i,n}\} _{n\in 	\N}$ is independent of $\{ m_{j,n}\} _{n\in \N}$ for $i\neq j$, $ m_{i,n}\sim N(0,1)$. Suppose that $\bfD _1,..., \bfD_l$ are diagonal, $\bfD _i=diag (a_{i,1},...,a_{i,d})$. Then the stationary solution $\bfX$ satisfies \eqref{affine} with 
		\begin{equation*}
		\bfA = diag (\sum _{i=1}^lm_{i}a _{i,1},...,\sum _{i=1}^lm_{i}a _{i,d}), 
		\end{equation*} 
		where $m_1,...,m_d$ are independent $N(0,1)$ variables. Therefore, diagonal entries 
		$A_{j}=\sum _{i=1}^lm_{i}a _{i,j}$ of $\bfA$ are normal variables with variance $\s _j=\sum _{i=1}^la^2 _{i,j}$. Then $|A_k|^{\a _k}= |A_j|^{\a _j}$ a.s. implies that
		$\a _j=\a _k$,
		\begin{equation*}
		|A_j|=|A_k| \ a.s.
		\end{equation*}
		and 
		\begin{equation*}
		\s _j=\sum _{i=1}^la^2 _{i,j}=\sum _{i=1}^la^2 _{i,k}=\s _k.
		\end{equation*}
		Therefore, for the blocks we have $\bfX^{(l)}$ we have
		\begin{equation*}
		\bfX^{(l)}=\bfA ^{(l)}\bfX^{(l)}+\bfB ^{(l)},
		\end{equation*}	
		with $\bfA ^{(l)}=A\ diag (\eps _1,...,\eps _{d_l})$, where $A$ is a positive random variable and $\eps _1,...,\eps _{d_l}$ are random variables with values $\pm 1$  i.e. $\bfA ^{(l)}$ is a similarity of the type considered in \cite{buraczewski:damek:guivarch2009}. If $|A_j|\neq |A_k|$ a.s. then the components $X_j$, $X_k$ of $\bfX$ are asymptotically independent.

\medskip
In both examples there is a further detailed description of measures $\nu _l$, see \cite{buraczewski:damek:guivarch2009} and \cite{mentemeier:wintenberger:2021}.  
\section{Homogeneous case}\label{homogeneous}
	Our aim in this section is to prove Theorem \ref{main} when for all $j$, $|A_j|^{\a _j}=|A_1|^{\a _1}$ a.s. i.e. in the homogeneous case. With the notation of the previous section, we have just one block and, for simplicity, we denote its dimension by $d$, not $d_l$.
	
	\begin{thm}\label{hom}
		Suppose that for every $j$, $|A_j|^{\a _j}=|A_1|^{\a _1}$ a.s. and  \eqref{affine1}- \eqref{A2},  are satisfied. 
		Then the sequence of measures $\lambda_t$ defined in \eqref{lambda} tends in a weak sense to a Radon measure $\Lambda $ on $\R ^d\setminus \{ 0\}$ 
		i.e.\begin{equation}\label{Lambda}
		\lim _{t\to \8} t\E f(\d _{t^{-1}}\bfX)= \langle f, \L \rangle . 
		\end{equation}
		for every $f\in \bfC (\R ^d )$. In particular, 
		\begin{equation*}
		c_{\8 }=\lim _{t\to \8}t\P \left ( |\bfX| _{\a }>t\right ).
		\end{equation*}
		Moreover, 
				\begin{equation}\label{Lambda1}
		\langle f\circ \d _t,\L \rangle = t\langle f ,\L \rangle ,
		\end{equation}
		$\L (\R ^d \setminus B_r(0))<\8$ for every ball $B_r(0)$,
		\begin{equation}\label{Lambda2}
		\L (S_r^{d-1})=0
		\end{equation} 
					and there is a finite measure $\wt \nu $ on $S^{d-1}$ such that
				\begin{equation}\label{Lambda3}
				\langle f, \L \rangle =  \int _0^{\8}\int _{S^{d-1}}f(\d _s\omega )\frac{ds}{s^2}d\wt \nu (\omega),\quad  \wt \nu (S^{d-1})=c_{\8} .
				\end{equation} 
				If $\bfX$ and $\bfB$ are independent then $\s $ in \eqref{A2} may be taken zero i.e. $\E |B_j|^{\a _j}<\8 $, for $j=1,...,d$, is sufficient.
			 
		Suppose additionally that $c_{\8}>0$. The sequence		
				 $m _t$ of probability measures on $S^{d-1}$ defined by
		\begin{equation}\label{Lambda4}
		m _t(U)= \P \left (\d _{|\bfX|^{-1}_{\a }}\bfX\in U \ \ \Big |\ \ |\bfX|_{\a } >t\right ), \quad \mbox{for a Borel}\ U\subset S^{d-1}.
		\end{equation} 
		tends weakly to $ \nu =c^{-1}_{\8}\wt \nu $ i.e. we may write 		
		\begin{equation}\label{Lambda5}
		\langle f, \L \rangle = c_{\8 } \int _0^{\8}\int _{S^{d-1}}f(\d _s\omega )\frac{ds}{s^2}d\nu (\omega),
		\end{equation}
	\end{thm}

		{\bf Remark.} Under assumption that $\bfX$ and $\bfB$ are independent and		
		$|A_1|=...=|A _d|$, Theorem \ref{hom} was proved in \cite[Theorem 2.8]{buraczewski:damek:guivarch2009}. It turns out that after some modifications, the proof follows in our setting. We just need to replace the euclidean norm, which is used in \cite{buraczewski:damek:guivarch2009}, by the norm $|\cdot | _{\a }$ as it is discussed in \cite[Appendix D]{buraczewski:damek:guivarch2009}. However, the details require a considerable amount of work and so they are presented here. 
		
		As in \cite{buraczewski:damek:guivarch2009} we start with 
		$\E f(\d _{t^{-1}}\bfX)$ 
		instead of using $\P (\d _{t^{-1}}\bfX\in \cdot )$ immediately. Then, some additional regularity (H\"older) of $f$ simplifies the proof and later on it may be gradually relaxed to arrive finally at \eqref{Lambda4}. 
	
	\begin{proof} 
		{\bf Step 1. Renewal equation. }Let $a=|A_1|^{\a _1}.$ Then 
		\begin{equation}\label{norm}
				|\bfA x| _{\a }=a |x| _{\a }= |\d _ax| _{\a }.\end{equation}
		Indeed, $|A_ix_i|^{\a _i}=|A_1|^{\a _1}|x_i|^{\a _i}$. \eqref{norm} means that then $\bfA$ acts by homogeneous dilations. 
		For $ f\in \bfC (\R ^d)$ we define
		\begin{equation*}
		\bar f(u) = \E f(\d _{e^{-u}}\bfX), \quad u\in \R
		\end{equation*}
		Let $\mu $ be a measure on $\R $ defined by 
		$$
		\mu (B)=\P (\log a \in B, a\neq 0)$$
		for a Borel set $B\subset \R $. Let 
		$\bar f * \mu (u):=\int _{\R }\bar f(u-s)d\mu (s)$ and
		\begin{equation}\label{renewal}
		\psi _f(u)=\bar f(u)-\bar f*\mu (u).
		\end{equation}
		Observe that if $A_i\geq 0$ for all $i$ then $a^{1\slash \a _i}=A_i$, $\d _a\bfX=A\bfX$ and 
		\begin{equation*}
		\E f(\d _{e^{u}}\bfX) \ d\mu (u)=\E f(A\bfX)\Ind {\{ a\neq 0\}}=\E f(A\bfX)
		\end{equation*}
		if $0\notin \supp f$. Therefore \eqref{renewal} becomes 
		\begin{equation}\label{renewal0}
		\psi _f(u)=\bar f(u)-\bar f*\mu (u)=\E f(\d _{e^{u}}\bfX)-\E f(\d _{e^{u}}A\bfX).
		\end{equation}
		\eqref{renewal0} is a renewal equation on $\R $ but unfortunately $\mu $ is either a strictly subprobability measure $(\P (A_1=0)>0))$ or, in view of \eqref{neglog}, a probability measure with strictly negative mean. Therefore, we change the measure multiplying both sides of the equation by $e^u$ and we have
		\begin{equation*}
		\wt \psi (u):=e^u\psi _f(u)=e^u\bar f(u)-e^u\bar f*\mu (u)=\wt f(u)-\wt f*\wt \mu (u),
		\end{equation*}
		where $\wt f(u)=e^u\bar f(u)$ and $\wt \mu = e^u\mu $. Now, in view of \eqref{A1} and \eqref{posder}, $\wt \mu $ is a probability measure on $\R $ with the positive mean $\bfm$. Therefore, we may apply the renewal theorem on $\R$, see Theorem \ref{renewalthm} in the Appendix, provided $\wt \psi $ is direct Riemann integrable. 
		
		However, the above argument does not work if $A_i$ may be negative because then $\d _a\bfX=(|A_1|X_1,...,|A_d|X_d)\neq A\bfX$.
		We have to write a renewal equation on the abelian group $D=\R \times \Z _2^n$ (for some $0\leq n\leq d$)  
		\begin{equation*}
		(u,k)(u_1,k_1)=(u+u_1, kk_1), 
		\end{equation*} 
		where $Z_2=\{ 1,-1\}$ with multiplication. $n$ is determined in the following way. Changing the order of coordinates, we may assume that for $1\leq i\leq n$, $\P (A_i <0)>0$, and for $n<i$, $\P (A_i <0)=0$. 
		The action of $k=(k_1,...,k_n)\in \Z _2^n$ on $\bfX$ is defined by $k\bfX=(k_1X_1,...,k_nX_n, X_{n+1},...,X_d)$ and for $g=(u,k)\in D$ we shall write
		\begin{equation*}
		\bar f(g)=\E f(\d _{e^{-u}}k\bfX).
		\end{equation*}
		Let $\mu $ be the measure on $D$ defined by
		\begin{equation*}
		\mu (U\times \{ k \})=\P (\log |A_1|^{\a _1}\in U, A_1\neq 0, \ \mbox{sign} A_i=k_i, i=1,...,n  ),
		\end{equation*}
		where $U$ is a Borel subset of $\R $. 
		(Observe, that $A_1\neq 0$ if and only if $A_i$ for every $i=1,...,d$). 
		Then 
		\begin{equation*}
		\int _D\psi (u,k)\ d\mu (u,k)= \E \psi (|A_1|^{\a _1}, \mbox{sign} A_1,..., \mbox{sign} A_n )
		\end{equation*}
		and 
		\begin{equation*}
		\E f(\d _{e^u}k\bfX)\ d\mu (u,k)=\E f(\bfA\bfX), 
		\end{equation*}
		because $a^{1\slash \a _i}\mbox{sign} A_i=A_i$. Finally writing $f*\mu (g)=\int _D\bar f(g(g')^{-1})\ d\mu (g')$ we obtain 
				\begin{equation}\label{renewal1}
		\psi _f(g)=\bar f(g)-\bar f*\mu (g)=\E f(\d _{e^{-u}}k^{-1}\bfX)-\E f(\d _{e^{-u}}k^{-1}\bfA\bfX)
		\end{equation}
		and the right renewal equation is
		\begin{equation*}
		\wt \psi (g):=e^u\psi _f(g)=e^u\bar f(g)-e^u\bar f*\mu (g)=\wt f(g)-\wt f*\wt \mu (g),
		\end{equation*}
		where $g=(u,k)$, $\wt f(u)=e^u\bar f(g)$ and $\wt \mu = e^u\mu $. As before, $\wt \mu $ is a probability measure on $\R $ with the positive mean $\bfm = \int _Du\ d\wt \mu (u,k)$.
				To apply the renewal theorem \ref{renewalthm} we have to show that
		\begin{equation}\label{renewal2}
		\wt f (g)=\sum _{m=0}^{\8} \wt \psi _f*\wt \mu ^m(g)
		\end{equation} 
		and that $\wt \psi _f$ is direct Riemann integrable ($dRi$) (as defined in the Appendix). 
		 Suppose we have \eqref{renewal2} and direct Riemann integrability. Then 
		\begin{equation*}
		\lim _{u\to \8}\wt f(g)=\frac{1}{\bfm}\int _D\wt \psi _f(u,k)\ dudk,
		\end{equation*}  
		which proves that 
		\begin{equation*}
		\lim _{t\to \8 }t\E f(\d _{t^{-1}}\bfX)\quad \mbox{exists}.\end{equation*}
		Here $du$ is the Lebesgue measure on $\R $ and $dk$ the probability measure equally distributed on elements of $\Z ^n$. 
		
		{\bf Step 2. Smoothing}
		However, although $\wt \psi$ is continuous, it not necessarily $dRi$ as it is. We need to assume more regularity of $f$. The key idea is to use first H\"older functions (see \cite{buraczewski:damek:guivarch2009}). For a $0<\zeta < 1$ let $H^{\zeta}$ be the space of functions $f\in \bfC (\R ^d) $ such that 
		\begin{equation}\label{holzeta}
		|f(x+y)-f(x)|\leq C_f|y|^{\zeta}_{\a}\quad \mbox{for a} \ C_f>0\ \mbox{and for all}\ x,y\in \R ^d.
		\end{equation}
		To define $H^{\zeta}$ we use the norm $|\cdot |_{\a}$ and it is not difficult to see that for a bounded function being in $H^{\zeta}$ is equivalent to be a standard H\"older function with a different exponent.
		For $f\in H^{\zeta }$, $\wt \psi _f$ is direct Riemann integrable and so
		\begin{equation}\label{convHol}
		\lim _{t\to \8 }t\E f(\d _{t^{-1}}\bfX)\ \mbox{exists for}\ f\in H^{\zeta}.\end{equation}
		Let's complete the details. By \eqref{renewal1}  
		\begin{equation*}
		\psi _f*\mu (g)=\bar f*\mu ^m(g)-\bar f*\mu ^{m+1} (g).
		\end{equation*}
		Therefore, 
		\begin{equation}\label{renowallim}
		\sum _{m=0}^N\psi _f*\mu (g)=\bar f (g)-\bar f*\mu ^{N+1} (g).
		\end{equation}
		Moreover, 
		\begin{equation}\label{renowalN}
		\lim _{N\to \8} \bar f*\mu ^{N+1} (g)=0.
		\end{equation}
		If $\P (a=0)\neq 0$ then the mass of $\mu $ is strictly less then one and \eqref{renowalN} holds. If $\P (a>0)=1$ then $\mu $ has a strictly negative mean, $f(0)=0$ and so by the strong law of large numbers
		\begin{equation*}
		\lim _{N\to \8} \bar f*\mu ^{N+1} (u)=\E f(\d _{e^{-t+s_1+...+s_{N+1}}}kk_1...k_{N+1}\bfX)d\mu (s_1,k_1)...d\mu (s_{N+1},k_{N+1})=0
		\end{equation*}
		Letting now $N\to \8$ in \eqref{renowallim}, we obtain
		\begin{equation*}
		\sum _{m=0}^{\8}\psi _f*\mu ^m(g)=\bar f (g)
		\end{equation*}   
		and so
		\begin{equation*}
		\wt f(g)=\sum _{m=0}^{\8}e^u\psi _f*\mu (g)=\sum _{m=0}^{\8}\wt \psi _f*\wt \mu (g).
		\end{equation*}   
		Now we are going to prove that, for $f\in H^{\zeta }$, $\zeta \geq \s$, the function $\wt \psi _f$ is $dRi$.
		
		Since $\wt \psi _f(u)=e^u\psi _f(g)$ is continuous, it is enough to prove that 
		\begin{equation}\label{rieman}
		I=\sum _{n\in \Z }\sup _{g\in \D _n}e^{u}|\psi _f(g)|<\8 ,
		\end{equation}
		where $\D _n=\{ g=(u,k) : n< u \leq n+1\}$.
		Let $\supp f\subset \{ x\in \R ^d: |x|_{\a}\geq \eta \}$. We have 
		\begin{align*}
		\psi _f(g)=&\E f(\d _{e^{-u}}k^{-1}\bfX)-\E f(\d _{e^{-(u-s)}}\left (k(k')^{-1})^{-1}\bfX \right )d\mu (s, k')\\
        &\E f(\d _{e^{-u}}k^{-1}\bfX)-\E f\left (\d _{e^{-u}}k^{-1}\d _{e^{s}}k'\bfX \right )d\mu (s,k')\\
		=&\E f(\d _{e^{-u}}k^{-1}(\bfA\bfX+\bfB))-\E f(\d _{e^{-u}}k^{-1}\bfA\bfX).
		\end{align*}
		Again here we make use of $f(0)=0$ and so a possible atom of $A_1$ at $0$ does not play any role.
		
		Hence for $g\in \Delta _n$, 
		\begin{equation*}
		e^u|\psi _f(g)|\leq e^{n+1}\E \left | f(\d _{e^{-u}}k^{-1}(\bfA\bfX+\bfB))- f(\d _{e^{-u}}k^{-1}\bfA\bfX)\right |\Ind {\{ |\bfA\bfX| _{\a}+|\bfB|_{\a}\geq e^n\eta c^{-1}_{\a }\}}.
		\end{equation*}
		Indeed, if
		\begin{equation}
		|\bfA\bfX| _{\a}+|\bfB|_{\a}<e^n \eta c^{-1}_{\a}
		\end{equation}
		then 
		\begin{equation*}
		|\bfA\bfX|_{\a}, |\bfA\bfX+\bfB|_{\a}\leq c_{\a}\left (|\bfA\bfX| _{\a}+|\bfB|_{\a} \right )\leq e^n\eta 
		\end{equation*}
		and both $f(\d ^{-u}k^{-1}(\bfA\bfX+\bfB)),f(\d ^{-u}k^{-1}\bfA\bfX)=0$. 
			
			Using the H\"older property of $f$, we have
		\begin{equation*}
		\sup _{u\in \Delta _n}   e^u|\psi _f(u)|\leq e^{n+1}e^{-n\zeta}\E  | \bfB| _{\a}^{\zeta }\Ind {\{ |\bfA\bfX| _{\a}+|\bfB|_{\a}\geq e^n\eta c^{-1}_{\a }\}}.
		\end{equation*}
		Let $P_n=\{ |\bfA\bfX| _{\a}+|\bfB|_{\a}\geq e^n\eta c^{-1} _{\a } \}$. Then
		\begin{align*}
		I\leq & e\sum _{n\in \Z}e^{n(1-\zeta)}\E |\bfB|^{\zeta}_{\a}\Ind {P_n}\\
		=&e\E \left (|\bfB|^{\zeta}_{\a}\sum _{n\leq n_0}e^{n(1-\zeta)}\right ),
		\end{align*}
		where $n_0$ is the random variable defined by 
		\begin{equation*}
		n_0=\log c_{\a }-\log \eta +\log \left ( |\bfA\bfX| _{\a}+|\bfB|_{\a}\right ).
		\end{equation*}
		Notice that if $n>n_0$ then $P_n=\emptyset$. Hence
		\begin{align*}
		I&\leq e^{1+(1-\zeta)(\log c-\log \eta)}\left (1-e^{\zeta -1}\right )^{-1}\E \left [ |\bfB|_{\a}^{\zeta}\left ( |\bfA\bfX| _{\a}+|\bfB|_{\a}\right )^{1-\zeta }\right ]\\
		&\leq C\left (\E |\bfB|_{\a}+\E \left [|\bfB|_{\a}^{\zeta} |\bfA\bfX| _{\a}^{1-\zeta }\right ]\right ). 
		\end{align*}
		It amounts to prove that
		\begin{equation*}
		\E \left [|\bfB|_{\a}^{\zeta} |\bfA\bfX| _{\a}^{1-\zeta }\right ]<\8 .
		\end{equation*}
		For, we apply the H\"older inequality with $p=\frac{1-\s ^2}{1-\zeta}$, $q=\frac{1-\s ^2}{\zeta -\s ^2}$, and so
		\begin{equation*}
		\E \left (|\bfB|_{\a}^{\zeta} |\bfA\bfX| _{\a}^{1-\zeta }\right )\leq 
		\left (\E |\bfB|_{\a}^{q\zeta}\right )^{1\slash q}\left (\E |\bfA\bfX| _{\a}^{1-\s ^2}\right )^{1\slash p}<\8 .
		\end{equation*} 
		because $q\zeta \leq 1+\s $ and 
		\begin{equation*}
		\E |\bfA\bfX| _{\a}^{1-\s ^2}= (\E a^{1-\s ^2})(\E |\bfX| _{\a}^{1-\s ^2})<\8 .
		\end{equation*}
		The last expectation is finite, because for every $i$, $\E |X_i|^{\a _i(1-\s ^2)}<\8$ in view of Lemma \ref{Rmoment}.
		
		\medskip
		 If $\bfX$ and $\bfB$ are independent then we proceed differently and we may take any $0<\zeta <1$. By \eqref{hom}, $|\bfA\bfX|_{\a }= |A_1|^{\a _1}|\bfX| _{\a}$ and so we have
		 \begin{equation*}
		 \E \left [|\bfB|_{\a}^{\zeta} |\bfA\bfX| _{\a}^{1-\zeta }\right ]=\E \left [|\bfB|_{\a}^{\zeta}|A_1|^{\a _1(1-\zeta )}\right ] \E |\bfX| _{\a}^{1-\zeta }.
		 \end{equation*}
		Finally, by H\"older inequality with $p=1\slash \zeta $, $q=1\slash (1-\zeta )$ we have
		\begin{equation*}
		\E \left (|\bfB|_{\a}^{\zeta}|A_1|^{\a _1(1-\zeta )}\right )\leq \left (\E |\bfB|_{\a}\right )^{\zeta}\left (\E |A_1|^{\a _1}\right )^{1-\zeta }=\left (\E |\bfB|_{\a}\right )^{\zeta}
		\end{equation*}

		{\bf Step 3. Existence of $\Lambda$.}
		To extend \eqref{Lambda} to $f\in C_c(\R ^d\setminus \{ 0\})$ and then to $f\in \bfC (\R ^d)$ we 
		use functions $h_r\in C^1(\R ^d)$ such that
		\begin{equation*}
		h _r(x)=\begin{cases} 0\quad x \in B_{r\slash 2}(0)\\
		1\quad x \in B_{r}(0)^c\end{cases}
		\end{equation*}
		and $0\leq h_r(x)\leq 1$.
				Then 
		\begin{equation}\label{ghol}
		h_r\in H^{\zeta}
		\end{equation}
		for $\zeta \leq (\max _{1\leq j\leq d}\a _j)^{-1}$, $\zeta <1$.
		To see \eqref{ghol} observe that for $h_r$ and $|y|_{\a}\geq 1$, \eqref{holzeta} is immediate, and for $|y|_{\a}\leq 1$ 
		\begin{equation*}
		| h_r(x+y)- h_r(x)|\leq C_h\| y\| \leq dC |y|_{\a}^{\zeta}.
		\end{equation*}
		Notice, that $\Ind {B_1(0)^c}$ is dominated by $h_1$ and so by \eqref{convHol}
				\begin{equation}\label{ind}
		\sup _{t>0}\P (|\bfX| _{\a}>t  )t\leq \sup _{t>0}t\E h_1(\d _{t^{-1}} \bfX)=C<\8. 
		\end{equation}
		Hence there is $C$ such that for every $r,t>0$
		\begin{equation*}
		\P (|\bfX| _{\a}>rt  )t\leq Cr^{-1}.
		\end{equation*}
		Moreover,
		\begin{equation}\label{estimball}
		 \L \left (B_{2r}(0)^c\right )\leq \langle \L , h_r\rangle = \lim _{t\to \8 }t\E h_r(\d _{t^{-1}} \bfX)\leq 
		 \sup _t t\P \left (|\bfX|_{\a }>tr\right )\leq Cr^{-1}.
		\end{equation}
			Now given $f\in C_c(\R ^d\setminus \{ 0\})$ and $\eps >0$, let $h'\in C_c^{1}(\R ^d\setminus \{ 0\})$ be such that $$\sup _{x\in \R^d}| f(x)-h'(x)| <\eps $$ and $\supp f, \supp h\subset B_r(0)^c$ for some $r>0$. Then
		\begin{equation}
		t\E \left | f(\d _{t^{-1}}\bfX)- h'(\d _{t^{-1}}\bfX) \right |\leq \eps t\P (|\bfX| _{\a}>t r )\leq C\eps r^{-1}
		\end{equation}
		which together with \eqref{convHol} proves that 
		\begin{equation}\label{tlimit}
		\lim _{t\to \8 }t^{-1}\E f(\d _t\bfX)\quad \mbox{exists}
		\end{equation}
		and \eqref{tlimit} implies existence of $\Lambda$.
		For generic $f\in \bfC (\R ^d)$ and large $M>0$, we write
		\begin{equation}\label{2r}
		f= h_{2M}f + \left (1-h_{2M}\right )f.
		\end{equation}
		Then 
		\begin{equation*}
		\lim _{t\to \8}t\E \left ((1-h_{2M})f\right )\left (\d _{t^{-1}}\bfX\right )=\langle \left (1-h_{2M}\right )f, \L \rangle .
		\end{equation*}
		Moreover,
		\begin{equation*}
		\langle h_{2M}f,\L \rangle \leq C\| f\| _{\8 }M^{-1}<\eps
		\end{equation*}
		and 
		\begin{align*}
		t\left |\E \left (h_{2M}f\right )\left (\d _{t^{-1}}\bfX\right )\right |\leq t\| f\| _{\8 }\P (|\bfX|_{\a}>tM)\leq 
		C\| f\| _{\8}M^{-1}<\eps
		\end{align*}
		if $M$ is sufficiently large and hence \eqref{Lambda} follows.
		
		\medskip
		{\bf Step 4. Properties of $\Lambda$.}
		\eqref{Lambda1} follows directly from definition.
		
		\medskip
		To obtain \eqref{Lambda2} for generic $f$ we proceed as in \cite{buraczewski:damek:guivarch2009}. 
		In view of \eqref{Lambda1} it is enough to prove \eqref{Lambda2} for $r=1$. For $0<\eps \leq \frac{1}{4}$, let 
		$\psi _{\eps }$ be the function on $\R $ defined by  
		\begin{equation*}
		\psi _{\eps} (u)= \begin{cases}
		1 \quad \quad \quad \quad \quad \quad \quad \mbox{if}\quad u=1\\
		 0 \quad \quad \quad \quad \quad \quad \quad   \mbox{if}\quad u\leq 1-\eps\ \mbox{or}\  u\geq 1+\eps \\
		u\slash \eps +1-1\slash \eps  \quad \quad \mbox{if}\quad 1-\eps \leq u\leq  1 \\ 
		 -u\slash \eps +1+1\slash \eps  \quad \mbox{if}\quad 1\leq u \leq  1+\eps .
		\end{cases}
		\end{equation*}
		We consider 
		\begin{equation}\label{phieps}
		\phi _{\eps }(x)=\sum _{j=1}^d\psi _{\eps}(|x_j|^{\a _j}).
		\end{equation}
		Then $\phi _{\eps }\in H^{\zeta}$ for $\zeta \leq (\max _{1\leq j\leq d}\a _j)^{-1}$, $\zeta <1$
		(see Appendix) and
		\begin{align*}
		\phi _{\eps }(\d _{t^{-1}}X)\neq 0 \ &\mbox{iff there is }\ j\ \mbox{such that}\ \psi _{\eps }\left (t^{-1}|X_j|^{\a _j}\right )\neq 0\\
		\mbox{i.e. there is }\ &j \ \mbox{such that}\ (1-\eps)t<|X_j|^{\a _j}< (1+\eps )t.
		\end{align*}
		Then
		\begin{equation*}
		t\E \phi _{\eps }(\d _{t^{-1}}\bfX)\leq \sum _{j=1}^dt\P \left ((1-\eps)t< |X_j|^{\a _j}\leq  (1+\eps)t\right ).
		\end{equation*}
		Letting $t\to \8 $, by Theorem \ref{Goldie} we obtain
		\begin{equation}\label{constants}
		\langle \phi _{\eps}, \L \rangle \leq \sum _{j=1}^dc_j\left ((1-\eps)^{-1}-(1+\eps)^{-1}\right ),
		\end{equation}
		where
		$c_j$ is the Goldie constant for $X_j$ and the right hand side of \eqref{constants} tends to $0$ when $\eps \to 0$. On the other hand, $\psi _{\eps}|_{S^{d-1}}\geq 1$, hence $\lim _{\eps \to 0}
		\langle \phi _{\eps}, \L \rangle \geq \L (S^{d-1})$ and so, \eqref{Lambda2} follows. 
		
		\medskip
		 To show \eqref{Lambda3}, for the random variable $\bfX$ we shall write
		$$
		\bfX=\d _{|\bfX|_{\a}}\o _\bfX =\Phi (|\bfX|_{\a}, \o _\bfX), \quad  \o_{\bfX}\in S^{d-1} $$
		in polar coordinates related to the dilations $\d _t$. Let $\phi \in C(S^{d-1})$ and 
		$f(s,\o )=f(\d _s\o)=\phi (\o )\Ind { B_1(0)^c}(\d _s\o)$. Then $\Lambda (\{ x: f \ \mbox{is not continuous at}\ x \}=0$ and by \eqref{ind} we obtain 
		\begin{equation*}
		|t\E f(\d _{t^{-1}}\bfX)|\leq t\E |\phi |(\o _\bfX)\Ind {|\bfX|_{\a }\geq t}\leq C|\phi |_{\8}.
		\end{equation*}
		Therefore, by the Portmanteau Theorem
		\begin{equation}\label{nu}
		\lim _{t\to \8}t\E f(\d _{t^{-1}}\bfX)=\langle \phi,\wt \nu \rangle 
		\end{equation}
		defines a bounded functional on $C(S^{d-1})$ and so a measure $\wt \nu $ on $S^{d-1}$. Moreover, for
		$f_r(\delta _s\o)=\phi (\o )\Ind { B_r(0)^c}(\delta _s\o)$, we have $f_r=f_1\circ \d _{r^{-1}}$,
		\begin{equation*}
		t\E f_r(\d _t^{-1}\bfX)=r^{-1}(rt)\E f_1(\d _{r^{-1}}\d _{t^{-1}}\bfX)
		\end{equation*}
		and so
		\begin{equation*}
		\langle f_r, \L \rangle =\lim _{t\to \8}t\E f_r(\d _t^{-1}\bfX)= r^{-1}\lim _{t\to \8}t\E f_1(\d _{t^{-1}}\bfX)=r^{-1}\langle \phi , \wt \nu \rangle .
		\end{equation*}
		Therefore,  
		\begin{equation*}
		\langle f_r, \L \rangle =r^{-1}\langle \phi , \wt \nu \rangle = \int _r^{\8}\frac{ds}{s^2}\langle \phi , \wt \nu \rangle = \int _{\R ^d}f_r(\d _s\o )\frac{ds }{s^2}d \wt \nu (\o) , 
		\end{equation*}
		which implies \eqref{Lambda3}. In particular, for $f=\Ind {B_1(0)^c}$ 
		\begin{equation*}
		t\E \Ind {B_1(0)^c}(\d _t^{-1}\bfX)=t\P (|\bfX| _{\a}>t)\to \wt \nu (S^{d-1})=c_{\8}.  
		\end{equation*}
		Finally, suppose that $c_{\8}>0$ and observe that for $\phi \in C(S^{d-1})$
	\begin{equation*}
	\langle \phi, m _t\rangle =\frac{\E [\phi (\o _\bfX)\Ind {\{ |\bfX|_{\a }>1\}}]}{\P (|\bfX|_{\a}>t)}\to \frac{\langle \phi,\wt \nu \rangle }{c_{\8}}.
	\end{equation*}
	\end{proof}
	\section {Asymptotic independence}\label{indep}
	In this section we shall prove that, if $|A_i|^{\a _i}\neq |A_j|^{\a _j}$ a.s then $X_i$ and $X_j$ are asymptotically independent i.e. 
	$$\P \left (|X_i|>t^{1\slash \a _i}, |X_j|>t^{1\slash \a _j}\right )=o(t^{-1}).$$
	
	\begin{thm}\label{asym1}
		Suppose that \eqref{affine1}-\eqref{A2} are satisfied and $\P \left (|A_i|^{\a _i}\neq |A_j|^{\a _j}\right )>0$. 
				Then, for every $r_1, r_2>0$ there are $C,\beta >0$ such that
		\begin{equation}\label{inddecay}
		\P \left (|X_i|>r_1t^{1\slash \a _i}, |X_j|>r_2t^{1\slash \a _j}\right )\leq Ct^{-1}(1+\log t)^{-\beta}\quad \mbox{for} \ t>1.
		\end{equation}
		In particular, for different blocks $\bfX^{(l)}, \bfX^{(q)}$ we have
		\begin{equation*}
		\P \left (|\bfX^{(l)}|_{\a }>r_1t, |\bfX^{(q)}|_{\a}>r_2t\right )\leq Ct^{-1}(1+\log t)^{-\beta}\quad \mbox{for} \ t>1.
		\end{equation*}
	\end{thm}
As usual, if $\bfX$ and $\bfB$ are independent then $\s $ in \eqref{A2} may be taken $0$ but then the conclusion is weaker.
\begin{thm}\label{asym3}
	Suppose that $\bfX$ and $B$ are independent, \eqref{A0}, \eqref{A1} are satisfied, 
	$$	\P \left (|A_i|^{\a _i}\neq |A_j|^{\a _j}\right )>0$$ and
	\begin{align*} 
	\E |A_i|^{\a _i}\log ^+|A_i|<\8 ,&\quad \E |B_i|^{\a _i}<\8, \\
	 \E |A_j|^{\a _j}\log ^+|A_j|<\8 ,&\quad \E |B_j|^{\a _j}<\8.
	\end{align*}
	Then, for every $r_1, r_2>0$,
	\begin{equation}
	\lim _{t\to \8}t\P \left (|X_i|>r_1t^{1\slash \a _i}, |X_j|>r_2t^{1\slash \a _j}\right )=0
	\end{equation}
	In particular, for different blocks $\bfX^{(l)}, \bfX^{(q)}$ we have
	\begin{equation*}
\lim _{t\to \8}	t\P \left (\bfX^{(l)}|_{\a }>r_1t, |\bfX^{(q)}| _{\a }>r_2t\right )=0.
	\end{equation*}
\end{thm}

With extra moment assumptions on $A_1, A_2$, we may improve the decay in \eqref{inddecay}. For that it is convenient to measure the growth of $A_1, A_2$ by submultiplicative functions. 

	\subsection{Submultiplicative functions}
	Let $G$ be a semi-group. A function $\tau :G\to [0,\8 )$ is called submultiplicative if
	\begin{equation*}
	\tau (gg')\leq \tau (g) \tau (g'), \quad \mbox{for}\ g,g'\in G.
	\end{equation*} 
	We are going to consider such functions on the multiplicative semi-groups $\R $, $\R^2$
	 and we shall always assume that $\tau $ is bounded
	on bounded sets i.e. $\tau $ is {\it locally bounded}. Typical examples on $\R $ are 
	\begin{equation*}
	\tau (g)=|g|^{\beta}, \ \beta >0, \quad \tau (g)=1+\log (1+|g|)
	\end{equation*}
	or
	\begin{equation*}
	\tau (g)=(1+\log (1+|g|))^{\beta}, \quad \tau (g)=|g|^{\beta _1} (1+\log (1+|g|))^{\beta _2}, \ \beta , \beta _1, \beta _2 >0.
	\end{equation*}
	More generally, if a function $\tau ':G \to [0,\8 )$ is subadditive i.e
	\begin{equation*}
	\tau '(gg')\leq \tau '(g) + \tau '(g')
	\end{equation*}
	then $\tau (g)=1+\tau '(g)$ is submultiplicative. Hence
	\begin{equation*}
	\tau (g)=1+\log (1+\log (1+|g|))
	\end{equation*}
	is submultiplicative and we may iterate like that. If $\tau _i$ are submultiplicative functions on $\R $ 
	then 
	\begin{equation*}
	\tau (g)=\tau (g_1,g_2)=\tau _1(g_1)\tau _2(g_2), \quad \mbox{for}\ g=(g_1,g_2)\in \R ^2
	\end{equation*} 
	is a submultiplicative function on $\R ^2$ considered with multiplication $$ (g_1,g_2)(g'_1,g'_2)=(g_1g'_1,g_2g'_2). $$ 
	But, of course there are other examples like
	\begin{equation*}
	\tau (g_1,g_2)=1+\log (1+|g_1g_2|)
	\end{equation*}
	or, more generally,
	\begin{equation*}
	\tau (g_1,g_2)=1+\tau _1 (|g_1g_2|).
	\end{equation*}
	Numerous other examples may be thought of as well.
	\begin{lem}
		Let $\tau :\R \to [0,\8 ) $ be a submultiplicative function. Then there are $C_1, C_2>0$   such that 
		\begin{equation}\label{tauestim}
		\tau (g)\leq C_1(1+|g|)^{C_2}.
		\end{equation}
	\end{lem}
\begin{proof}
	Let $g>e$ and $C=\max(\sup _{1\leq w\leq e}\tau (w), 1)$. Then 
	\begin{align*}
	\tau (g) =&\tau \left (e^{\lfloor \log g\rfloor}g^{\log g - \lfloor \log g\rfloor}\right )\\
	\leq & \tau \left (e^{\lfloor \log g\rfloor}\right )\sup _{1\leq w\leq e}\tau (w)\\
	\leq &C\tau (e)^{\lfloor \log g\rfloor}\leq C g^{\log C}.
	\end{align*}
	Moreover, $\tau (-|g|)\leq \tau (-1)\tau (|g|)$, $\tau $ is bounded on $[-e,e]$ and so \eqref{tauestim} follows.
\end{proof}
\subsection{Improvement of Theorem \ref{asym1}}
\hskip 15 pt Let $\tau $ be a submultiplicative function on $\R ^2$. We shall write $\tau (g)=\tau (g_1, g_2)$ and \newline $\tau (g_1):=\tau (g_1,1)$, $\tau (g_2):=\tau (1, g_2)$ identifying $g_1, g_2$ with $(g_1,1)$ and $(1,g_2)$ respectively. Moreover, we shall assume that $\tau (g_1,g_2)>0$ if $(g_1,g_2)\in (\R \setminus \{ 0\})^2$.
Now we may generalize Theorem \ref{asym1}. Without loss of generality we may consider $X_1$ and $X_2$ instead of $X_i$, $X_j$.
\begin{thm}\label{asym2}
	Suppose that \eqref{affine1}-\eqref{A2} are satisfied, $\P \left (|A_1|^{\a _1}\neq |A_2|^{\a _2}\right )>0$ 
	and that there is a locally bounded submultiplicative function $\tau $ on $G=\R ^2$ 
	such that
	\begin{equation}\label{condtau}
	\E |A_1|^{\a _1} \tau (A_1)< \8, \quad \E |A_2|^{\a _2} \tau (A_2)< \8
	\end{equation}
	and 
		$$\tau (g_1,g_2)>0\quad  \mbox{if}\quad  g_1, g_2\neq 0.$$
			Then, for every $r_1, r_2>0$ there are $\beta ,C>0$ such that
	\begin{equation}\label{extdecay}
	\P \left (|X_1|>r_1t^{1\slash \a _1}, |X_2|>r_2t^{1\slash \a _2}\right )\leq Ct^{-1}
	\tau (t ^{1\slash \a _1},t ^{1\slash \a _2})^{-\beta}.
	\end{equation}
\end{thm}
	{\bf Remark.} Observe, that we may obtain an extra logarithmic decay if $\tau (g_1,g_2)= (1+\log |g_1|)^{M_1}(1+\log |g_2|)^{M_2}$ or an extra polynomial decay if $\tau (g_1,g_2)= |g_1|^{\beta _1} |g_2|^{\beta _2}$. In Theorem \ref{asym1},
	$\tau (g_1,g_2)= (1+\log |g_1|)(1+\log |g_2|)$. As before, we have \eqref{extdecay} also for different blocks.
\subsection{Proof of Theorem \ref{asym2}}
In this section we consider the multiplicative group $G=(\R \setminus \{ 0\})^2$. Let $\eta $ be a measure on $G$ defined by 
\begin{equation*}
\eta (W_1\times W_2)=\P (A_1\in W_1,A_2\in W_2)
\end{equation*}
for Borel sets $W_1, W_2\subset \R \setminus \{ 0\}$. 
\begin{lem}
	Given $0<\xi <1$ let
	\begin{equation}\label{subpr}
	\chi _{\xi}(g) = |g_1|^{\a _1\xi }|g_2|^{\a _2(1-\xi )}.
	\end{equation}
	Then $\wt \eta _{\xi }=\chi _{\xi }(g)\eta $ is a strictly subprobability measure. Moreover, if \eqref{condtau} holds then there is $\g _0>0$ such that 
	\begin{equation}\label{subpro}
\int _G \tau (g)^{\g}\ d\wt \eta _{\xi }(g)<1, \quad \mbox{for} \ \gamma \leq \gamma _0.
	\end{equation}
\end{lem}
\begin{proof} Since $|A_1|^{\a _1}\neq |A_2|^{\a _2}$, we have
	\begin{align*}
	\int _G\chi _{\xi }(g)\ d\eta (g)&=\int _G |g_1|^{\a _1\xi }|g_2|^{\a _2(1-\xi )} \ d \eta (g) =\E |A_1|^{\a _1\xi }|A_2|^{\a _2(1-\xi) }\\
	&< \left (\E |A_1|^{\a _1} \right )^{\xi }\left (\E |A_2|^{\a _2} \right )^{1-\xi }=1,
	\end{align*}
	which proves that $\eta _{\xi }$ is a strictly probability measure.
	Let 
	\begin{equation}\label{weaker}
	k (s)=\E \tau (A_1,A_2)^s|A_1|^{\a _1\xi }|A_2|^{\a _2(1-\xi) }.
	\end{equation}
	Then
	\begin{align*}
	k (s)&\leq \E \tau (A_1)^s\tau (A_2)^s|A_1|^{\a _1\xi }|A_2|^{\a _2(1-\xi) }\\
	&\leq \left (\E \tau (A_1)^{s\slash \xi } |A_1|^{\a _1}\right )^{\xi }\left (\E \tau (A_2)^{s\slash (1-\xi)}|A_2|^{\a _2}\right )^{1-\xi}<\8 
		\end{align*}
	provided $s< s_0=\min (\xi , (1-\xi))$. Moreover, $k(0)<1$ and $k(s) $ is continuous on $[0,s_0)$. Hence \eqref{subpro} follows.	
\end{proof}
	{\bf Remark.} For the proof of Theorem \ref{asym2} we need just one $\xi $, possibly $\xi =1\slash 2$ but we keep a more general setting because in some specific cases the choice of $\xi $ may be used to improve or optimize \eqref{extdecay}. Moreover, \eqref{condtau} may be replaced by \eqref{weaker} with $s<\min (\xi , 1-\xi )$ which is slightly weaker. 

\medskip
{\it Proof of Theorem \ref{asym2}.}
	 Let $f\in H^{\zeta}(\R ^2)$ be such that $\supp f\subset \{ x: |x_1|\geq \eps _1,\ |x_2|\geq \eps _2 \}$ and  $X=(X_1,X_2)$. Define
\begin{equation}\label{h}
h(g)=\E f(g^{-1}X):= \E f(g_1^{-1}X_1,g_2^{-1}X_2),\quad \wt h(g)=\chi _{\xi}(g)h(g)=\chi _{\xi}(g) \E f(g^{-1}X) . 
\end{equation}
Observe that contrary to the previous section we don't use dilations in \eqref{h} (as it was for $\bar f$ in the previous section) but just multiplication. 
$\wt h$ is bounded. Indeed,
\begin{align*}
\chi (g) \E f(g^{-1}X)&\leq |g_1|^{\a _1\xi }|g_2|^{\a _2(1-\xi )}||f| |_{sup}\E \Ind {|X_1|>\eps _1g_1}
\Ind {|X_2|>\eps _2g_2}\\
&\leq |g_1|^{\a _1\xi }|g_2|^{\a _2(1-\xi )}||f|| _{sup}\left (\E \Ind {|X_1|>\eps _1g_1}\right )^{\xi}
\left (\E \Ind {|X_2|>\eps _2g_2}\right )^{1-\xi }\leq C.
\end{align*}
Let $h*\eta (g)=\int _G h(g(g')^{-1})d\eta (g')$ i.e. here $*$ is the convolution on $G$.  
\begin{equation}\label{psiasym}
\psi (g)=h(g)-h*\eta (g)\quad \mbox{and}\quad \wt \psi (g)=\chi _{\xi}(g)\psi (g).
\end{equation} 
Then
\begin{align*}
\wt \psi (g)=&\wt h(g)-\wt h * \wt \eta _{\xi}(g)\\
\wt \psi *\wt \eta _{\xi}^m(g)=&\wt h*\wt \eta _{\xi}^m(g)-\wt h * \wt \eta _{\xi}^{m+1}(g).
\end{align*}
Hence 
\begin{equation*}
\sum _{m=0}^N\wt \psi *\wt \eta _{\xi}^m(g)=\wt h(g)-\wt h * \wt \eta _{\xi}^{N+1}(g)
\end{equation*}
and $\lim _{N\to \8}\wt h*\wt \eta _{\xi}^{N+1}(g)=0$ because $\wt \eta _{\xi}$ is a strictly subprobability measure.
Therefore,
\begin{equation*}
\wt h (g)=\sum _{m=0}^{\8}\wt \psi *\wt \eta _{\xi}^m(g)=\wt \psi *U _{\xi}(g),
\end{equation*}
where $U _{\xi}=\sum _{m=0}^{\8 }\wt \eta ^m$ is a finite measure.
We shall prove in Lemma \ref{psitau1} that for some $\beta ,C>0$
\begin{equation}\label{htau}
\wt h(g)=\wt \psi *U _{\xi}(g)\leq C\tau (g)^{-\beta}.
\end{equation}
Now in \eqref{h} we take a non negative $f\in H^{\zeta}(\R ^2)$ such that $f(x)=1$ if $|x_1|\geq r_1, |x_2|\geq r_2$ and $\supp f\subset \{ x: |x_1|\geq r _1\slash 2,  \ |x_2|\geq r _2 \slash 2\}$. For $g=(t^{1\slash \a _1}, t^{1\slash \a _2})$, by \eqref{htau}, we have
\begin{align*}
\P \left (|X_1|>r_1t^{1\slash \a _1}, |X_2|>r_2t^{1\slash \a _2}\right )&\leq h \left (t^{1\slash \a _1}, t^{1\slash \a _2}\right )\\
= \chi _{\xi } \left (t^{1\slash \a _1}, t^{1\slash \a _2}\right) ^{-1}\wt h \left (t^{1\slash \a _1}, t^{1\slash \a _2}\right )&\leq Ct^{-1}\tau \left (t^{1\slash \a _1}, t^{1\slash \a _2}\right )^{-\beta  }
\end{align*}
and \eqref{extdecay} follows.\hskip 400 pt \boxed{}

\bigskip
To complete \eqref{htau} first we need the following lemma
\begin{lem}\label{psitau}
	Suppose that $f$ and $\xi $ are fixed, $\wt \psi $ is defined in \eqref{psiasym}. Then there are $\g _1, C >0$ such that
	\begin{equation*}
	\wt \psi (g)\leq C\tau (g)^{-\g}, \quad \mbox{for}\ \gamma \leq \gamma _1.
	\end{equation*}
\end{lem}
\begin{proof}
	Recall that $X=(X_1,X_2)$ satisfies $X\stackrel{d}{=}AX+B$ with $A=diag (A_1,A_2), B=(B_1,B_2)$. In view of \eqref{h} we have
	\begin{align*}
	\psi (g)=&\E f(g^{-1}X)-\E f(g^{-1}g'X) d\eta (g')\\
	=&\E f(g^{-1}X)-\E f(g^{-1}AX) \\
	=&\E f(g^{-1}(AX+B))-\E f(g^{-1}AX).
	\end{align*}
	In the second equality we use that $0\notin \supp f$. Observe that if for $i=1$ or $i=2$
	\begin{equation}\label{supp}
	|A_iX_i|+|B_i|<\eps _i|g_i|
	\end{equation}
	then $\E f(g^{-1}(AX+B))-\E f(g^{-1}AX)=0$. Indeed, \eqref{supp} for given $i$ implies that
	\begin{equation*}
	|g_i^{-1}(A_iX_i+B_i)|\leq |g_i^{-1}A_iX_i|+|g_i^{-1}B_i|<\eps _i
	\end{equation*}
	and
	\begin{equation*}
	|g_i^{-1}A_iX_i|< \eps _i.
	\end{equation*}
	Hence $f(g^{-1}(AX+B))=f(g^{-1}AX)=0$. 
	Therefore,
	\begin{equation}\label{cut}
	\E | f(g^{-1}(AX+B))-f(g^{-1}AX)|\leq C_f\E \| g^{-1}B\| ^{\zeta }_{\a}\Ind {|A_1X_1|+|B_1|\geq \eps _1|g_1|}\Ind {|A_2X_2|+|B_2|\geq \eps _2|g_2|}.
	\end{equation}
	We may assume that $\| g^{-1}B\| ^{\zeta }_{\a }=| g_1^{-1}B_1| ^{\a _1\zeta }$. The other case is analogous. Let $0<\wt \zeta <1$. We have 
	\begin{align*}
	\E | f(g^{-1}(AX+B))-&f(g^{-1}AX)|\leq C_f\E | g_1^{-1}B_1| ^{\a _1\zeta }\Ind {|A_2X_2|+|B_2|\geq \eps _2|g_2|}\\
	&\leq C_f\E | g_1^{-1}B_1| ^{\a _1\zeta }(|A_2X_2|+|B_2|)^{\a _2(1-\wt \zeta )} 
	\left (\eps _2|g_2|\right )^{-\a _2(1-\wt \zeta)}\\
	&= C_f \eps _2^{-\a _2(1-\wt \zeta)} | g_1|^{-\a _1\zeta }|g_2|^{-\a _2(1-\wt \zeta)}\E \left [ | B_1| ^{\a _1\zeta }(|A_2X_2|+|B_2|)^{\a _2(1-\wt \zeta )} \right ]
\end{align*}
Finally,
\begin{align*}
\tau (g)^{\g }\wt \psi (g)&=\tau (g)^{\g} \chi _{\xi }(g) \psi (g)\\
&\leq C\tau (g)^{\g } | g_1|^{\a _1(\xi -\zeta )}|g_2|^{\a _2(\wt \zeta -\xi)}\left (\E \left [ | B_1| ^{\a _1\zeta }|A_2X_2| ^{\a _2(1-\wt \zeta )}\right ]+\E \left [|B_1| ^{\a _1\zeta }|B_2|^{\a _2(1-\wt \zeta )} \right ]\right ).
\end{align*}
First we prove that if 
\begin{equation}\label{zeta}
\zeta (1+\s \a _1^{-1})^{-1}< \wt \zeta <1
\end{equation}
then 
\begin{equation}\label{twoexp}
  \E \left [ | B_1| ^{\a _1\zeta }|B_2|^{\a _2(1-\wt \zeta )} \right ], \E \left [ | B_1| ^{\a _1\zeta }|A_2X_2|^{\a _2(1-\wt \zeta )} \right ]<\8 .
\end{equation}
Indeed, to estimate the second term we choose 
$$\frac{1}{p}=\frac{\zeta }{1+\s \a _1^{-1}}, \quad \frac{1}{q}=1-\frac{\zeta }{1+\s \a _1^{-1} }>1-\wt \zeta .$$
Then
\begin{equation*}
\E \left [ | B_1| ^{\a _1\zeta }|A_2X_2|^{\a _2(1-\wt \zeta )} \right ]\leq 
\left (\E \left [ | B_1| ^{\a _1\zeta p} \right ]\right )^{1\slash p}
\left (\E \left [ |A_2X_2|^{\a _2(1-\wt \zeta )q } \right ]\right )^{1\slash q}<\8
\end{equation*}
because $\a _1\zeta p= \a _1+\s$ and $\a _2(1-\wt \zeta )q<\a _2$. In the same way we estimate the first term in \eqref{twoexp}.

Now we claim that $\tau (g)^{\g }\wt \psi (g)$ is bounded provided $\g $ is small enough. Observe that  by \eqref{tauestim}, there is $s $ such that 
\begin{equation*}
\tau (g)\leq \tau (g_1,1)\tau (1,g_2)\leq C(1+|g_1|)^{s }(1+|g_2|)^{s }.
\end{equation*}
Moreover, $\psi $ is bounded. Hence, if $|g_1|, |g_2|\leq 1$ then $\tau (g)^{\g }\chi _{\xi}(g) \psi (g)$ is bounded as well. 
If $|g_1|>1$,  $|g_2|\leq 1$, we chose $\zeta >\xi, \wt \zeta\geq \xi $. Then 
\begin{equation*}
\tau (g)^{\g}| g_1|^{\a _1(\xi -\zeta )}|g_2|^{\a _2(\wt \zeta-\xi)}\leq C(1+|g_1|)^{\g s}|g_1|^{\a _1(\xi -\zeta )},
\end{equation*}
which is bounded provided $\g s \leq \a _1(\zeta -\xi )$.
If $|g_2|>1, |g_1|\leq 1$, we chose $\wt \zeta <\xi, \zeta \leq \xi $. Then 
\begin{equation*}
\tau (g)^{\g}| g_1|^{\a _1(\xi -\zeta )}|g_2|^{\a _2(\wt \zeta-\xi)}\leq C(1+|g_2|)^{\g s}|g_2|^{\a _2(\wt \zeta -\xi  )},
\end{equation*}
which is bounded provided $\g s \leq \a _2(\xi -\wt \zeta )$.
If both $|g_1|, |g_2|> 1$, we chose $\wt \zeta <\xi < \zeta $. Then 
\begin{equation*}
\tau (g)^{\g}| g_1|^{\a _1(\xi -\zeta )}|g_2|^{\a _2(\wt \zeta-\xi)}\leq C(1+|g_1|)^{\g s}|g_1|^{\a _1(\xi -\zeta )}(1+|g_2|)^{\g s}|g_2|^{\a _2(\wt \zeta -\xi  )},
\end{equation*}
which is bounded provided $\g s \leq \a _1(\zeta -\xi )$, $\g s \leq \a _2(\xi -\wt \zeta )$. In all the cases we keep \eqref{zeta} and so, in view of \eqref{twoexp} the conclusion follows.
\end{proof}
\begin{lem}\label{psitau1}
	Let $\g \leq \min (\g _0, \g _1)$. Then there is $C$ such that
	\begin{equation*}
	\wt \psi *U_{\xi}(g)\leq C\tau (g)^{-\g}.
	\end{equation*}
\end{lem}
\begin{proof}
	Let $\g \leq \g _0$. Then by \eqref{subpro}
	\begin{align*}
	&\int _G\tau (g)^{\gamma}\ dU_{\xi}(g)=\sum _{n=0}^{\8}\int _G\tau (g)^{\g}d \wt \eta _{\xi}^n(g)\\
	=&\sum _{n=0}^{\8}\int _G\tau (g_1...g_n)^{\g}\ d \wt \eta _{\xi}(g_1)...\ d \wt \eta _{\xi}(g_n)
	\leq \sum _{n=0}^{\8}\left (\int _G\tau (g)^{\g}\ d \wt \eta _{\xi}(g)\right )^n<\8.
	\end{align*}
	By Lemma \ref{psitau}, for $\g \leq \min (\g _0, \g _1)$, we have 
	\begin{equation*}
	\wt \psi *U_{\xi }(g')=\int _G\wt \psi (g'g^{-1})\ dU_{\xi}(g)\leq C\int _G \tau (g'g^{-1})^{-\g}\ dU_{\xi}(g).
	\end{equation*}
	But $\tau (g')\leq \tau (g'g^{-1}) \tau (g)$. Hence
	\begin{equation*}
	\tau (g'g^{-1})^{-\g}\leq \tau (g')^{-\g} \tau (g)^{\g}
	\end{equation*}
	an so
	\begin{equation*}
	\wt \psi *U_{\xi}(g')\leq C \tau (g')^{-\g }\int _G \tau (g)^{\g}\ dU_{\xi}(g).
	\end{equation*}
\end{proof}

\subsection{Proof of Theorem \ref{asym3}}
In this subsection we assume that $X, B$ are independent. We keep notation of the previous section.
In particular, $f$ is a non negative function in $H^{\zeta}(\R ^2)$ such that  $f(x)=1$ if $|x_1|\geq r _1$,   $|x_2|\geq r _2$,
$\supp f\subset \{ x: |x_1|\geq r _1\slash 2,  \ |x_2|\geq r _2 \slash 2\}$. Then $\psi ,\wt \psi $ are defined in \eqref{psiasym}, $\wt h $ in \eqref{h}.

 For $n,m\in \Z$ let
\begin{equation*}
\Delta_{n,m}=\{ g: n<g_1\leq n+1, m<g_2\leq m+1\}.
\end{equation*}
\begin{lem}Suppose that \eqref{A0}-\eqref{A1} are satisfied, $\E |A_i|^{\a _i}\log ^+|A_i|<\8 $ and $\E |B_i|^{\a _i}<\8$. Let $\zeta <\xi $. Then
	\begin{equation}\label{Iriem}
	I=\sum _{n,m}\sup _{g\in \Delta _{n,m}}|\chi _{\xi}(g)\psi (g)|<\8
	\end{equation}
\end{lem}
Once the above lemma is proved, Theorem \ref{asym3} follows. Indeed, 
\begin{align*}
t\P \left (|X_1|>r_1t^{1\slash \a _1}, |X_2|>r_2t^{1\slash \a _2}\right )&\leq \wt h \left (t^{1\slash \a _1}, t^{1\slash \a _2}\right )
= \wt \psi *U_{\xi}(t^{1\slash \a _1}, t^{1\slash \a _2})\\
&=\int _G\wt \psi (t^{1\slash \a _1}g_1^{-1}, t^{1\slash \a _2}g_2^{-1} )U_{\xi}(g)\to 0, \ \mbox{as}\ t\to \8 .
\end{align*}
because, in view of \eqref{Iriem}, $\lim _{t\to \8} \wt \psi (t^{1\slash \a _1}g_1^{-1}, t^{1\slash \a _2}g_2^{-1} )=0$ and $U_{\xi}$ is a finite measure.

\medskip
{\it Proof of Theorem \ref{asym3}.}
	We choose $0<\zeta <\xi $. As in the proof of Lemma \ref{psitau}, 
	\begin{equation*}
	\psi(g)\leq C_f\E \| g^{-1}B\| ^{\zeta }_{\a}\Ind {|A_1X_1|+|B_1|\geq \eps _1|g_1|}\Ind {|A_2X_2|+|B_2|\geq \eps _2|g_2|}.
		\end{equation*}
		We may assume that $\| g^{-1}B\| ^{\zeta }_{\a}=|g_1|^{-\a _1\zeta}| B_1| ^{\a _1\zeta }$. The other case is analogous. If
		$g\in \Delta_{n,m}$ then
		\begin{equation*}
		\wt \psi (g)\leq C_f e^{(n+1)\a _1\xi +(m+1)\a _2(1-\xi)}e^{-n\a _1\zeta } \E | B_1| ^{\a _1\zeta }_{\a}\Ind {|A_1X_1|+|B_1|\geq \eps _1e^n}\Ind {|A_2X_2|+|B_2|\geq \eps _2e^m}
		\end{equation*}
		 Let $P_{n,m}=\{|A_1X_1|+|B_1|\geq \eps _1e^n, |A_2X_2|+|B_2|\geq \eps _2e^m \}$. Then
		\begin{align*}
		I\leq & e^{\a _1\xi + \a _2(1-\xi)}\sum _{n,m\in \Z ^2}e^{n\a _1(\xi-\zeta)}e^{m\a _2(1-\xi)}\E |B_1|^{\a _1\zeta}\Ind {P_{n,m}}\\
		=&e^{\a _1\xi + \a _2(1-\xi)}\E \left (|B_1|^{\a _1\zeta}\sum _{n\leq n_0, m\leq m_0}e^{n\a _1(\xi-\zeta)}e^{m\a _2(1-\xi)}	\right )	
		\end{align*}
		where $n_0, m_0$ are random variables defined by 
		\begin{equation*}
		n_0=-\log \eps _1 +\log \left ( |A_1X_1| +|B_1|\right ), \quad m_0=-\log \eps _2 +\log \left ( |A_2X_2| +|B_2|\right ).
		\end{equation*}
		Indeed, if $n>n_0$ or $m>m_0$ then $P_{n,m}=\emptyset$. Hence for\newline
		$C_0=e^{\a _1\xi + \a _2(1-\xi)}\left (1-e^{\a _1(\xi -\zeta)}\right )^{-1}
		\left (1-e^{\a _2(1-\xi )}\right )^{-1}\eps _1^{-\a _1(\xi -\zeta)}\eps _2^{-\a _2(1-\xi)}$
		\begin{align*}
		I&\leq C_0\E |B_1|^{\a _1\zeta}\left ( |A_1X_1| +|B_1|\right )^{\a _1(\xi -\zeta )}
		\left ( |A_2X_2| +|B_2|\right )^{\a _2(1-\xi)}\\
		&\leq C_0 C  \Big (\left (\E |B_1|^{\a _1\zeta}|A_1|^{\a _1(\xi -\zeta)}|A_2| ^{\a _2(1-\xi) }\right ) \left (\E |X_1|^{\a _1(\xi -\zeta)}|X_2| ^{\a _2(1-\xi) }\right )\\
		&+\left (\E |B_1|^{\a _1\zeta}|A_1|^{\a _1(\xi -\zeta)}|B_2| ^{\a _2(1-\xi) }\right ) \left (\E |X_1|^{\a _1(\xi -\zeta)}\right )\\
		&+\left (\E |B_1|^{\a _1\zeta}|B_1|^{\a _1(\xi -\zeta)}|A_2| ^{\a _2(1-\xi) }\right ) \left (\E |X_2| ^{\a _2(1-\xi) }\right )\\ 
		&+\left (\E |B_1|^{\a _1\xi}|B_2| ^{\a _2(1-\xi) }\right ) \Big)
		\end{align*}
		It amounts to prove that all the expectations above are finite. We shall use the H\"older inequality with various exponents. First take $1\slash p= \xi -\zeta \slash 2, \  1\slash q= 1-\xi +\zeta \slash 2$ and so
		\begin{equation*}
		\E |X_1|^{\a _1(\xi -\zeta)}|X_2| ^{\a _2(1-\xi) }\leq \left (\E |X_1|^{\a _1(\xi -\zeta)p}\right )^{1\slash p}\left (\E |X_2| ^{\a _2(1-\xi)q }\right )^{1\slash q}<\8
		\end{equation*} 
		because $\a _1(\xi -\zeta)p <\a _1, \a _2(1-\xi)q <\a _2$. For the last term we take $p=1\slash \xi , q=1\slash (1-\xi)$. For the remaining expectations we take $r_1=1\slash \zeta ,  r_2=1\slash (\xi -\zeta), r_3=1\slash (1-\xi)$ i.e
		\begin{equation*}
		\E |B_1|^{\a _1\zeta}|A_1|^{\a _1(\xi -\zeta)}|B_2| ^{\a _2(1-\xi) }\leq \left (\E |B_1|^{\a _1}\right )^{\zeta}\left (\E |A_1|^{\a _1}\right )^{(\xi -\zeta)}\left (\E |B_2| ^{\a _2}\right )^{(1-\xi) }<\8 .
		\end{equation*}
			 
\section{Proof of Theorem \ref{main} and \ref{mainth}}\label{conclusion}
The main step is to prove \eqref{main1}. The argument follows closely the one in \cite{buraczewski:damek:mirek:2010}, Lemma 2.6. It is rather standard and it generalizes the one in Lemma 2.1 in \cite{resnick:1987}. See also \cite{davis:resnick:1996}.
Using continuous functions in \eqref{main1} or \eqref{main2} instead of passing immediately to $m _t$  considerably simplifies the arguments. 

Let $\L _l$ be the tail measure obtained in Theorem \ref{hom} and 
\begin{equation*}
B_r^{(l)}(0)=\{ x\in \R ^{(l)}: |x|_{\a} <r \}.
\end{equation*}
To prove \eqref{main1} observe that every function $f\in \bfC (\R ^d)$ may be written as a sum of $p$ functions
\begin{equation*}
f= f^{(1)}+...+f^{(p)}
\end{equation*}
such that
\begin{equation*}
\supp f^{(l)}\subset \R^{d_1}\times ...\R ^{d_{l-1}}\times \left (\R ^d_{l}\setminus B_{r}^{(l)}(0)\right)\times \R ^{d_{l+1}}\times...\times \R ^{d_p},
\end{equation*}
and so, without loss of generality, it is sufficient to prove that
\begin{equation*}
\lim _{t\to \8} t\E f^{(1)}(\d _{t^{-1}}\bfX)=\langle f^{(1)}_1,\L _1\rangle ,
\end{equation*} 
where $f_1^{(1)}(x_1)=f(x_1,0,...,0)$. 
To simplify the notation, we will write $f$ instead of $f^{(1)}$, $f_1$ instead of $f^{(1)}_1$ and we assume that
\begin{equation}\label{kula}
\supp f\subset \left (\R ^{d_{1}}\setminus B_{r}^{(1)}(0)\right )\times \R ^{d_2}\times...\times \R ^{d_p}.\end{equation}

Suppose first that $f\in H^{\zeta }(\R ^{d})$ and $f_1(x_1)=f(x_1,0,...,0)$. Then $f_1\in H^{\zeta }(\R ^ {d_1})$ and so by \eqref{convHol} 
\begin{equation}
\lim _{t\to \8}t\E f_1(\d _{t^{-1}}\bfX ^{(1)}=\langle f_1,\L _1\rangle .
\end{equation} 
Now we write $\bfX=(\bfX^{(1)},\bar \bfX)$, $\bar \bfX= (\bfX^{(2)},...,\bfX^{(p)})$ and consider 
\begin{align*}
I_t:=&t \left |\E f (\d _{t^{-1}}\bfX) - \E f_1(\d _{t^{-1}}\bfX^{(1)})\right |\leq t \E \left |f (\d _{t^{-1}}\bfX) - f(\d _{t^{-1}}\bfX^{(1)}, 0)\right | \\
=&\E \left |f(\d _{t^{-1}}\bfX^{(1)}, \d _{t^{-1}}\bar \bfX)-f (\d _{t^{-1}}\bfX^{(1)},0)  \right | \Ind {\{|\bar \bfX|_{\a}>\eps t\}}\\
+&\E \left |f(\d _{t^{-1}}\bfX^{(1)}, \d _{t^{-1}}\bar \bfX)-f (\d _{t^{-1}}\bfX^{(1)},0)\right | \Ind {\{|\bar \bfX|_{\a}\leq \eps t\}}=: I_{t,1}+I_{t,2}.
\end{align*}
Now we are going to prove that $I_t\to 0$ when $t\to \8$. For $I_{t,1}$, by Theorem \ref{asym1} or \ref{asym3}, we have
\begin{equation*}
|I_{t,1}|\leq t\| f \| _{\8}\E \Ind {\{|\bfX ^{(1)}|_{\a}> rt\}}\Ind {\{|\bar \bfX|_{\a}>\eps t\}}= C \| f\| _{\8 }o(1), \quad \mbox{as}\ t\to \8.
\end{equation*}
$I_{t,2}$ may be estimated as follows
\begin{align*}
|I_{t,2}|&\leq t\E \left |f(\d _{t^{-1}}\bfX^{(1)}, \d _{t^{-1}}\bar \bfX)-f (\d _{t^{-1}}\bfX^{(1)},0)\right | \Ind {\{| \bfX ^{(1)}|_{\a}> rt\}}\Ind {\{|\bar \bfX|_{\a}\leq \eps t\}}\\
&\leq tC_{f}\E |\d _{t^{-1}}\bar \bfX|_{\a}^{\zeta}\Ind {\{|\bfX ^{(1)}|_{\a}>t\}}\leq tC_f\eps ^{\zeta}\P (|\bfX^{(1)}|_{\a}> t)\leq C\eps ^{\zeta}.
\end{align*}
Finally, for every $\eps >0$
\begin{equation*}
\limsup _{t\to \8} I_t\leq C\eps ^{\zeta }
\end{equation*}
and, letting $\eps \to 0$, we obtain the conclusion.

For generic $f\in \bfC (\R ^d)$ we proceed as in the proof of Theorem \ref{hom}. For a large $M$, as in \eqref{2r}, we write
\begin{equation*}
f=h_{2M}f+\left (1-h_{2M}\right )f
\end{equation*}
and observe that 
\begin{align*}
\left |\E h_{2M}f(\d _{t^{-1}}\bfX)\right |&\leq t\|f\|_{\8 }\P \left (|\d _{t^{-1}}\bfX|_{\a }>M\right )\\
&\leq t\|f\| _{\8 }\sum _{l=1}^p\P \left (|\bfX^{(l)}|_{\a }>Mt\right )\leq C\|f\| _{\8}M^{-1}
\end{align*}
and 
\begin{equation*}
\left |\langle (h_{2M}f)_1, \L _1 \rangle \right |\leq |f| _{\8}\L _1\left (B_M^{(1)}(0)^c \right )\leq C|f| _{\8}M^{-1}.
\end{equation*}
It is now enough to prove that 
\begin{equation}\label{compact}
\lim _{t\to \8}t\E \left ((1-h_{2M})f\right ) (\d _{t^{-1}}\bfX)=\langle \left ((1-h_{2M})f\right )_1, \L _1\rangle
\end{equation}
and pass with $M$ to $\8 $. $\left (1-h_{2M}\right )f\in C_c(\R ^d)$ and so, given $\eps >0$, it may be approximated by $g\in H^{\zeta }(\R ^d)$ such that $\supp \ g\subset \left (\R^{d_1}\setminus B^{(1)}_{r\slash 2}\right ) \times \R ^{d_2}\times ...\times \R ^{d_p}$ (see \eqref{kula})
and
\begin{equation*}
\left \|(1-h_{2M})f-g\right \| _{\8 }<\eps .
\end{equation*}
Then also
\begin{equation*}
\left \|\left ((1-h_{2M})f\right )_1-g_1\right \| _{\8 }<\eps
\end{equation*}
and, by \eqref{estimball}
\begin{equation*}
\left | \langle \left ((1-h_{2M})f\right )_1, \L _1\rangle -\langle g_1, \L _1\rangle \right | <\eps \L _1\left (B_{r\slash 2}^{(1)}(0)^c \right )\leq C\eps r^{-1}. 
\end{equation*}
Moreover, 
\begin{equation*}
\lim _{t\to \8}t\E  g(\d _{t^{-1}}\bfX)=\langle g_1, \L _1\rangle .
\end{equation*}
Hence
\begin{align*}
\limsup _{t\to \8}&\left |t\E \left (1-h_{2M}\right )f (\d _{t^{-1}}\bfX)-\langle \left ((1-h_{2M})f\right )_1, \L _1\rangle \right |\\
&\leq \limsup _{t\to \8}\left |t\E \left (1-h_{2M}\right )f (\d _{t^{-1}}\bfX)-t\E g(\d _{t^{-1}})\right |
+ \lim _{t\to \8}\left |t\E g(\d _{t^{-1}})-\langle g_1, \L _1\rangle\right |\\
&+\left |\langle g_1, \L _1\rangle -\langle \left ((1-h_{2M})f\right )_1, \L _1\rangle \right |\\
&\leq \eps t \P \left ( |\bfX^{(1)}|_{\a}>rt \right )+C\eps r^{-1}\leq C\eps r^{-1}
\end{align*}
Now letting $\eps \to 0$, we obtain \eqref{compact} and \eqref{main1} is proved.

\eqref{main3},\eqref{main2} follow from Theorem \ref{hom}. Moreover, by Theorem \ref{asym1}  ``blocks'' $\bfX^{(l)}$ are  asymptotically independent. 
Hence to prove \eqref{main5} 
we proceed as in the proof of Theorem 6.1 in \cite{mentemeier:wintenberger:2021}. Finally, to 
obtain Theorem \ref{mainth} we use \eqref{Lambda5} and again the arguments contained in \cite{mentemeier:wintenberger:2021}, proof of Theorem 6.1.

\section{One dimensional case}\label{onedim}
In this section we consider the one-dimensional case.  
	First, for the readers convenience, we recall the Goldie Implicit Renewal Theorem. $R,M$ are random variables with values in $\R $.
	\begin{thm}[\cite{goldie:1991}]
		\label{Goldie}
			Suppose that 	
		there exists $\alpha>0$ such that\\
		\quad\, $\bullet$ $\E[|M|^\alpha]=1$,\,  $\E |M|^\alpha \log^+|M|<\infty$. \\
		$\bullet$\  
		and the conditional law of $\log |M|$ given
		$\{M\neq 0\}$ is non-arithmetic. Then
		\begin{equation*}
		-\8 \leq \E \log |M| <0
		\end{equation*}
		and
		\begin{equation*}
		0< \bfm =\E |M|^{\a}\log |M| <\8.
		\end{equation*}
		Suppose further that $R,M$ are independent and
				\begin{equation}\label{implicite}
		\int _0^{\8}|\P (\pm R>t)-\P (\pm MR>t)|t^{\a -1}\ dt <\8 
		\end{equation}
		Then there exist constants $c_\pm$ such that 
		\begin{align*}
		\lim _{t\to \8}\P(\pm R>t) t^{\a}= 	c_\pm 
		\end{align*}
		as $x\to\infty$. 
		\end{thm}
		{\bf Remark.} Without further assumptions $c_+, c_-$ may be zero.
Suppose now that $R \stackrel{d}{=} MR+Q$, $M,R$ are independent, but $R$ is not necessarily independent of $Q$. 
As an immediate application of Theorem \ref{Goldie} we obtain the following asymptotics of $\P (\pm R>t)$. 
	\begin{thm}
		\label{Goldie1}
				Suppose that $M$ satisfies the assumptions of the previous theorem, 	
		 $R,M$ are independent, 
		\begin{align}
		\label{def:uni:sfe}
		R \stackrel{d}{=} MR+Q,
				 		\end{align}
		and there is $\s >0$ such that
		\begin{equation}\label{Qmom}
		\E[|Q|^{\alpha +\s}]<\infty .
		\end{equation}
		Then there exist constants $c_\pm$ such that 
		\begin{align*}
		\lim _{t\to \8}\P(\pm R>t) t^{\a}= 	c_\pm 
		\end{align*}
		as $x\to\infty$. 
	\end{thm}
		
				{\bf Remark.} When $R$ is independent of $(M,Q)$, Theorem \ref{Goldie1} was proved in \cite{goldie:1991} 
					with $\s =0$ and only a slight modification is needed to conclude it in our setting.  
	
\begin{proof}
	    Theorem \ref{Goldie1} follows from Lemma 2.2 and Theorem 2.3 in \cite{goldie:1991}. 
    Only condition \eqref{implicite}
        must be checked and under \eqref{def:uni:sfe} it becomes
        \begin{equation*}
       \int _0^{\8}|\P (\pm (MR+Q)>t)-\P (\pm MR>t)|t^{\a -1}<\8.
        \end{equation*}
        The above follows from 
    \begin{equation*}
    I=\E \left | \left ((MR+Q)^{\pm}\right )^{\a }-\left ((MR)^{\pm}\right )^{\a }\right |<\8 ,
    \end{equation*}
    as it is explained in Lemma 9.4 in \cite{goldie:1991}. Indeed, if $\a \leq 1$ then 
    $$I\leq \left | (MR+Q)^{\pm}-(MR)^{\pm}\right |^{\a }\leq \E |Q|^{\a }<\8$$ 
    because for any $a,b\in \R $, $|a^{\pm}-b^{\pm}|\leq |a-b|$.
    If $\a >1$, we may write
    \begin{align*}
    I\leq &\a \E \left [ | (MR+Q)^{\pm}-(MR)^{\pm} | \max (|MR+Q|^{\a -1}, |MR|^{\a -1})\right ]\\
    &\a \E \left [ |Q|\max (|MR+Q|^{\a -1}, |MR|^{\a -1})\right ]\\
    \leq &\a \max(1, 2^{\a -2})\E \left [ |Q| (|Q|^{\a -1} + |MR|^{\a -1})\right ].
    \end{align*}
    Now we have to prove that $\E \left [|Q||MR|^{\a -1}\right ]<\8 $.
   Taking $p=\a +\s , q=(\a +\s )\slash (\a +\s -1)$ and using independence of $M,R$, we obtain
    \begin{equation*}
    \E \left [|Q||MR|^{\a -1}\right ]\leq \left (\E |Q|^{\a +\s}\right )^{1\slash p}
\left (\E |M|^{q(\a -1)}\right )^{1\slash q} \left (\E |R|^{q(\a -1)}\right )^{1\slash q}<\8
    \end{equation*}
    because $q(\a -1)<\a $ and so $\E |R|^{q(\a -1)}<\8 $ in view of the next lemma. 
\end{proof}
	\begin{lem}\label{Rmoment}
		Suppose that \eqref{def:uni:sfe} and the assumptions of Theorem \ref{Goldie} are satisfied,
		$$R \stackrel{d}{=} MR+Q,$$
		 and $\E |Q|^{\a}<\8$. Then for every $0<s <\a $
		\begin{equation}\label{Rmom}
	\E |R|^s<\8.
	\end{equation}
	\end{lem}
	{\bf Remark.} When $R$ and $(M,Q)$ are independent, \eqref{Rmom} follows immediately from the representation $R\stackrel{d}{=}\sum _{i=0}^{\8} M_1...M_{i-1}Q_i$, $(M_i,Q_i)$ being i.i.d with the same law as $(M,Q)$, \cite{goldie:1991}.
\begin{proof}
Let $f(t)=\P (|R|>t)$. We have to prove that $\int _0^{\8} t^{s-1}f(t)\ dt <\8$. Let $\mu $ be the the measure on $\R ^+$ defined by $\mu (B) = \P (|M|\in B)$, for a Borel set $B\subset \R ^+$. Observe that
\begin{equation*}
\E f (tM^{-1})=\int _{\R ^+}\P (R>ts^{-1})\ d\mu (s)=f*\mu (s),
\end{equation*}	
where $*$ denotes the convolution on the multiplicative group $\R ^+$. Let 
$$\psi (t)=f(t)-f*\mu (t)= \P (|R|>t)-\P (|MR|>t).$$ Then
\begin{equation*}
f (t)= \psi (t) + f*\mu (t)
\end{equation*}
and so 
\begin{equation*}
f (t)=\sum _{m=0}^N \psi *\mu ^m(t) + f*\mu ^{N+1}(t).
\end{equation*}
But $f*\mu ^{N+1}(t)=\P (|M_1...M_{N+1}|>t)\to 0 $ when $N\to \8$. Indeed, $M_1,...,M_{N+1}$ are i.i.d, and $\mu $ is either a strictly subprobability measure (if $\P (|M|=0)>0$) or it is a probability measure with $\E \log |M|<0$. Therefore,  
\begin{equation}\label{series} 
f (t)=\sum _{m=0}^{\8} \psi *\mu ^m(t).
\end{equation}
Now multiplying both sides of \eqref{series} by $t^s$,\ $s<\a $ and writing $\wt f(t)= t^sf(t), \ \wt \psi (t)= t^s \psi (t),\linebreak  \wt \mu =t^s \mu $ we have
\begin{equation*}
\wt f = \sum _{m=0}^{\8}\wt \psi *\wt \mu ^m=:\wt \psi *\wt U,
\end{equation*}
where $\wt U=\sum _{m=0}^{\8}\wt \mu ^m$ is a finite measure, because $\wt (\R ^+)<1$. Since 
\begin{align*}
\int _0^{\8}\wt f(t)t^{-1}\ dt=&\int _0^{\8} \wt \psi *\wt U(t)t^{-1}\ dt =\int _0^{\8} \wt \psi (ts^{-1})\ d\wt U(s)t^{-1}\ dt\\
=&\int _0^{\8} \wt \psi (ts^{-1})t^{-1}\ dt \ d\wt U(s)= \wt U(\R )\int _0^{\8} \wt \psi (t)t^{-1}\ dt .
\end{align*}
Now it amounts to prove that
\begin{equation}\label{integrab}
\int _0^{\8} \wt \psi (t) t^{-1}\ dt <\8 .
\end{equation}
 Since
\begin{equation*}
\wt \psi (t)= t^{s}\left (\P (|MR+Q|>t)-\P (|MR|>t)\right ),
\end{equation*}
by Lemma 9.4 in \cite{goldie:1991}, \eqref{integrab} follows from
\begin{equation*}
I_s =\E \left ||MR+Q|^s - |MR|^s|\right |<\8 .
\end{equation*}
For $s\leq 1$, the integrand is dominated by $|Q|^s$, so $I_s<\8$. Now suppose that $\a >1$ and $1\leq s _0=\sup \{ s: I_s<\8 \}<\a $. Take
$s_0<s<s_0+1-s_0\slash \a $. Then
\begin{equation*}
I_s\leq s\max(1,2^{s-2})\E |Q| (|MR|^{s-1}+|Q|^{s-1})
\end{equation*}
and we need to prove that $\E |Q||MR|^{s-1}<\8$. Proceeding as always, by H\"older inequality with $p=\a , q=\a \slash (\a -1)$, we have 
$$\E |Q||MR|^{s-1}\leq 
\left (\E |Q|^{\a }\right )^{1\slash \a}
\left (\E |M|^{q(s -1)}\right )^{1\slash q} \left (\E |R|^{q(s -1)}\right )^{1\slash q}<\8$$
because $q(s-1)<s_0.$
\end{proof}
\subsection{Strict positivity of $c_++c_-$}\label{positivity}
Asymptotics
\begin{equation}\label{strictpos}
t^{\a}\P (|R|>t)\sim c_++c_-
\end{equation}
  is exact if $c_++c_->0$. If $R$ and $(M,Q)$ are independent then  $c_++c_->0$ if and only if for every $x\in \R $,
 $\P(Mx+Q= x)<1$ i.e. $R$ is not constant, as it is shown in \cite{goldie:1991}, Theorem 4.1.  
Moreover, if $\P (M<0)>0$ then $c_+=c_->0$.
		
		If $\P (M\geq 0)=1$, $c_+, c_-$ may not be equal and so we may ask when $c_+, c_->0$. Conditions for strict positivity of $c_+, c_-$ in the case $\P (M\geq 0)=1$ were first given in \cite{guivarch:lepage:2012} and then elaborated and simplified in \cite{buraczewski:damek:mikosch:2016}, \cite{buraczewski:damek:2017}. It turns out that $c_+>0$ if and only if $R^+=\max (0,R)$ is unbounded.
		In particular, it may happen that $c_+>0$ and $c_-=0$.
	
If $R$ and $Q$ may be dependent none of the above methods works and then complex analysis helps. We may proceed as in \cite{buraczewski:damek:guivarch2009}. In view of the observations made in \cite{goldie:1991}, if \eqref{implicite} and \eqref{def:uni:sfe} are satisfied then 
\begin{equation}
c_++c_-=\int _0^{\8}\left (\P ( |MR+Q|>t)-\P ( |MR|>t)\right )t^{\a -1}\ dt 
\end{equation}
and so under assumptions of Theorem \ref{Goldie1}, in view of Lemma 9.4 in Goldie
\begin{equation*}
c_++c_-=\frac{1}{\a m}\E \left (|MR+Q|^{\a}-|MR|^{\a }\right ).
\end{equation*}	
Let
\begin{equation*}
s_{\8}=\sup \{ s: \E |M|^s + \E |Q|^s <\8 \}
\end{equation*}
and for $z\in \C $ we write
 $$\kappa (z):=\E |M|^z.$$
\begin{thm}\label{Rcontra}
	Suppose that $s_{\8}>\a $. Then $c_++c_-=0$ if and only if for every $s<s_{\8 }$
	\begin{equation}\label{Rallmom}
	\E |R| ^{s}<\8.
	\end{equation}
\end{thm}
	{\bf Remark.} It is reasonable to ask why and when \eqref{Rallmom} is not possible and it gives a contradiction. It is so, for instance if
	\begin{equation}\label{condR}
	s_{\8}=\8 \ \mbox{and}\ \sup _{s< s_\8}\left (\frac{\E |Q|^s}{\kappa (s)}\right )<\8\quad  \mbox{or}\quad  
	s_{\8}<\8 \ \mbox{and}\ \lim  _{s\to s_\8}\frac{\E |Q|^s}{\kappa (s)}=0
	\end{equation}
	as it is proved in \cite{buraczewski:damek:guivarch2009}. The proof in \cite{buraczewski:damek:guivarch2009} does not use independence of $Q$ and $R $, so we may conclude that $c_++c_-$ is strictly positive if assumptions of Theorem \ref{Goldie1} and \eqref{condR} are satisfied. In various settings, \eqref{Rallmom} may be impossible for some specific reasons depending on the model.
\begin{proof}
	Suppose that $c_++c_-=0$ (the other direction is obvious by \eqref{strictpos}). Let $\theta =\sup \{ s: \E |R|^s<\8 \}$. 
	
{ \bf Step 1.} First we shall prove that $\theta >\a $. For $z\in \C $, $0\leq \Re z $ we consider 
\begin{equation}\label{Guiv}
\L (z)=\E \left (|MR+Q|^z-|MR|^z\right ).
\end{equation}
If $\Re z <\a $ then $\L $ is well defined and
\begin{equation*}
\L (z) = \E |R|^z -\E |MR|^z= (1-\kappa (z))\E |R|^z.
\end{equation*}
Observe that $\Lambda $ is well defined for $\Re z=\a +\d$, $\d \leq \frac{\s }{\a +\s}$, $\s$ as in \eqref{Qmom}. Indeed, if $\Re z=\a +\d >1$ then
\begin{equation*}
|a^z-b^z|=\left |\int _a^bzr^{z-1}\ dr\right |\leq |z||a-b|\max \left (a^{\Re z -1}, b^{\Re z -1}\right ).
\end{equation*}
Hence, with $a=|MR|$ and $b=|MR+Q|$, we have
\begin{equation*}
|\Lambda (z)|\leq |z|\max\left (1,2^{\a +\d -2}\right )\E \left ( |Q|^{\a +\d}+ |Q||MR|^{\a +\d -1}\right )<\8.
\end{equation*}
Indeed, as before, for $ \E |Q||MR|^{\a +\d -1}$ we use H\"older inequality with $p=\a +\s , \ q=\frac{\a +\s}{\a +\s -1} $. Then $(\a +\d -1)q<\a $ and so
\begin{equation*}
\E |Q||MR|^{\a +\d -1}\leq \left (\E |Q|^{\a +\s}\right )^{1\slash p}\left (\E |MR|^{(\a +\d -1)q}\right )^{1\slash q}<\8 .
\end{equation*}
If $0\leq \R z\leq 1$ then
\begin{equation}\label{subcom}
|a^z-b^z|\leq (2|z|+3)|a-b|^{\Re z},
\end{equation}
(see the proof below) and so 
\begin{equation*}
|\Lambda (z)|\leq (2|z|+3)\E |Q|^{\Re z}<\8.
\end{equation*}

{\bf Step 2.} Now we shall conclude \eqref{Rallmom}. Since $\kappa '(\a )=\E |M|^{\a}\log |M|>0$, the function
\begin{equation*}
h(z)=\frac{\Lambda (z)}{1-\kappa (z)}
\end{equation*}
is well defined and holomorphic in a neighborhood of $\a $. Moreover, for $\Re z<\a $,
\begin{equation*}
h(z)=\E |R| ^{z}.
\end{equation*}
In view, of Landau Theorem (Theorem \ref{landau}) the abscissa of convergence $\theta $ of the law of $|R|$ is strictly larger then $\a $.
Now we shall prove that for every $s<s_{\8 }$
\begin{equation*}
\E |R|^s<\8 .
\end{equation*}
Suppose that $\theta <s_{\8}$ and consider $\theta \leq s\leq \theta +\d <s _{\8}$, $\d \leq \frac{\s }{\theta +\s}$. Proceeding as in Step 1 we prove that
\begin{equation*}
|\Lambda (s)|\leq \begin{cases} \E |Q|^s <\8 \hskip 140 pt \mbox{if}\ s\leq 1\\ s\max (1,2^{s-2})\E \left [|Q|^s+|Q||MR|^{s-1}\right ]\quad \mbox{if}\ s> 1.\end{cases}
\end{equation*}
Since, in the second case, $0<s-1<\theta $, applying the H\"older inequality as before we obtain that $|\L (s)| <\8$, which, by contradiction, proves that $\theta \geq s_{\8}$.

{\bf Proof of \eqref{subcom}.} Finally, to prove \eqref{subcom}, observe that for $a\geq b \geq 0, a>0$ , $u=ba^{-1}$ and $z=s+i\beta $
\begin{equation*}
|a ^z-b^z|=|a^z|| 1-u^z|=a^s|1-u^z|
\end{equation*}
so it is enough to prove that 
\begin{equation}\label{sub}
|1-u^z|\leq (2|z|+3)|1-u|^s, \ \mbox{for}\  0\leq u\leq 1.
\end{equation}
If $0\leq u\leq 1\slash  2$ then 
\begin{equation*}
|1-u^z|\leq |1-u|^s + u^s|1-u^{i\beta}|\leq |1-u|^s + 2u^s\leq 3 |1-u|^s.
\end{equation*}
If $1\slash 2<u\leq 1$ then
\begin{equation*}
|1-u^z|\leq \int _u^1 |z|r^{s-1}\ dr\leq |z||1-u|2^{1-s}\leq 2|z||1-u|^s.
\end{equation*}
\end{proof}
\section{Appendix}
	Let $D=\R $ or $D=\R \times \Z ^n_2$ be a direct product of $\R$ and $\Z ^n_2$. 
	We say that a function $ \psi $ defined on $D$ is direct Riemann integrable ($dRi$) if it is continuous and satisfies
	\begin{equation}\label{dri}
	\sum _{n\in \Z}\sup _{(u,k)\in \Delta _n}| \psi (u,k)|<\8 ,
	\end{equation}
	$\Delta _n=\{ (u,k): n\leq u <n+1 \}$.
	
	The above condition is a little bit stronger then the standard definition of $dRi$ (see \cite{revuz:1975}, Chapter 5) but it is sufficient for us and it is simpler to check.
	\begin{thm}\label{renewalthm}(\cite{revuz:1975})
		Let $\mu $ be a probability measure on the group $D$ such that $\supp \mu $ generates $D$ and
		\begin{equation*}
		\bfm=\int _Du \ d\mu (u,k)>0.
		\end{equation*}
		Then for every d-R-i function $\psi $ on $D$ we have
		\begin{equation*}
		\lim _{u\to \8} \psi *U(u,k)=\frac{1}{\bfm}\int _D \psi (u,k)du dk.
		\end{equation*}
	\end{thm}
For the proof we refer to \cite{athreya:mcdonald:ney:1978} or \cite{revuz:1975}. For further explanations see also Appendix A in \cite{buraczewski:damek:mikosch:2016}. 

We shall recall also a lemma due to Landau. We consider a measure $\g $ defined on $\R ^+$ and its Mellin transform $\hat \g (s)=\int _{\R ^+}x^s\ d\g (x), s>0$. 

\begin{lem}\label{landau}
	Suppose that $\hat \g $ is well defined for $s<\theta $ and $\hat \g (s)=\8 $ for $s>\theta $. Then $\hat \g $ canot be extended holomorphically to a neighborhood of $\theta $.  
	\end{lem} 
$\theta $ is called {\it the abscissa of convergence}.
  
  \medskip
The rest of the Appendix contains two standard analytical computations proving properties used before.

	\begin{proof}[Proof of \eqref{polar}] 
		$\Phi $ is continuous and any $x$ in $\R ^d$ may be written as
		\begin{equation*}
		x=\d _{|x|_{\a }}\left (\d _{|x|^{-1}_{\a }}x\right )\quad \mbox{ with}\quad \d _{|x|^{-1}}x\in S^{d-1}.
		\end{equation*}
		To show that $\Phi $ is 1-1, suppose that 
		$$
		\d _{s_1}\o _1=\d _{s_2}\o _2\quad \mbox{and}\quad |\o _{1}|_{\a}=1,\ |\o _{2}|_{\a }=1.$$
		Then 
		\begin{equation*}
		s_2=|\d _{s_2}\o _{2}|_{\a}=|\d _{s_1}\o _{1}|_{\a}= s_1.
		\end{equation*} 
		and so $\o _1=\o _2$. Hence 
		\begin{equation*}
		\Phi ^{-1}(x)=(|x|_{\a},\d _{|x|^{-1}_{\a }}x ).
		\end{equation*}
		Finally, we prove that $\Phi ^{-1}$ is continuous. Indeed, let $x_n=\d _{|x _n| _{\a}}\o _n\to x=\d _{|x |_{\a }}\o\neq 0$. Then $|x_n|\to |x|$ and so 
		$\o _n =\d _{|x _n|^{-1}} x_n\to \o = \d _{|x |^{-1}}x$.  
	\end{proof}
	\begin{lem}
		$\phi _{\eps}$ defined \eqref{phieps} belongs to $H^{\zeta} $ for any  $\zeta \leq (\max _{1\leq j\leq d}\a _j)^{-1}$, 
		$\zeta <1$.
	\end{lem}
	\begin{proof}
		Since $\phi _{\eps}$ is bounded, it is enough to prove that
		\begin{equation*}
		I=|\phi _{\eps }(x+h)-\phi _{\eps}(x)|\leq C|h|_{\a}^{\zeta}
		\end{equation*}
		for $|h|_{\a}\leq 1$. But
		\begin{equation*}
		I\leq \sum _{j=1}^d\left |\psi _{\eps }(|x_j+h_j|^{\a _j})-\psi _{\eps }(|x_j|^{\a _j})\right |
		\end{equation*}
		and $\psi _{\eps }(|x_j+h_j|^{\a _j})-\psi _{\eps }(|x_j|^{\a _j})=0 $ for $|x_j|> (1+\eps)^{1\slash \a _j}+1$ and $|h_j|^{\a _j}\leq 1$. Moreover, $\psi _{\eps}$ is Lipschitz. So for every $j$ we have
		\begin{equation*}
		|	\psi _{\eps }(|x_j+h_j|^{\a _j})-\psi _{\eps }(|x_j|^{\a _j})|\leq C_{\eps} 
		\left ||x_j+h_j|^{\a _j}-|x_j|^{\a _j} \right |.
		\end{equation*}
		For $\a _j \leq 1$,
		\begin{equation*}
		\left ||x_j+h_j|^{\a _j}-|x_j|^{\a _j} \right | \leq |h_j|^{\a _j}\leq |h|_{\a}\leq |h|_{\a}^{\zeta},
		\end{equation*}
		and for $\a _j > 1$,
		\begin{equation*}
		\left ||x_j+h_j|^{\a _j}-|x_j|^{\a _j} \right | \leq C|h_j|\leq |h|_{\a}^{1\slash \a _j}\leq |h|_{\a}^{\zeta}
		\end{equation*}
		because  $x_j$ are bounded and $ |h|_{\a}\leq 1$.
	\end{proof}

\noindent {\bf Acknowledgments.}  The author would like to thank Tomas Kojar for pointing the connection between the Gaussian multiplicative chaos and stochastic difference equations.

	The work was supported by the NCN grant 2019/33/B/ST1/00207.

\end{document}